\documentclass[12pt]{amsart}
\usepackage{hyperref}
\usepackage{indentfirst}
\usepackage{appendix}
\usepackage{amsmath,amsthm,amsfonts,amssymb,stmaryrd,mathrsfs,enumitem}
\usepackage{latexsym}
\usepackage{a4,color,palatino,fancyhdr}
\usepackage{graphicx}
\usepackage{float}  
\usepackage[small, bf, margin=90pt, tableposition=bottom]{caption}
\RequirePackage{amssymb}
\RequirePackage[T1]{fontenc}
\usepackage{amscd}
\usepackage{xypic}
\input{xy}
\xyoption{all}
\xyoption{poly}
\usepackage[all]{xy}
\usepackage{tikz}
\usepackage{graphicx}

\usepackage{amsmath}
\usepackage{tikz}
\usepackage{mathtools}

\newtheorem{theorem}{Theorem}[section]
\newtheorem{proposition}[theorem]{Proposition}
\newtheorem{definition}[theorem]{Definition}
\newtheorem{lemma}[theorem]{Lemma}
\newtheorem{corollary}[theorem]{Corollary}
\newtheorem{problem}{Problem}
\newtheorem{remark}{Remark}
\newtheorem{example}{Example}[section]

\numberwithin{equation}{section}

\def\cl#1{{\langle #1 \rangle}}

\def\mod{{\mathrm {mod}}}

\newcommand\N{{\mathbb{N}}}
\newcommand\Z{{\mathbb{Z}}}
\newcommand\Q{{\mathbb{Q}}}

\newcommand\R{{\mathbb{R}}}
\newcommand\CC{{\mathbb{C}}}

\newcommand\poincare{Poincar\'{e}}

\newcommand\bI{\mathbf{I}}

\newcommand {\be}{\begin{equation}}
\newcommand {\ee}{\end{equation}}

\newcommand\cooloverbrace[2]{\mathrlap{\smash{\overbrace{\phantom{\begin{matrix}#2\end{matrix}}}^{\mbox{$#1$}}}}#2}

\newcommand\coolrightbrace[2]{\left.\vphantom{\begin{matrix}#1\end{matrix}}\right\}#2}

\begin{document}

\title{Twisted Milnor hypersurfaces I.}

 \author{Jingfang Lian }
\address{School of Mathematical Sciences, Fudan University, Shanghai, 200433, P.R. China. }
\email{18110180033@fudan.edu.cn}

\author{Fei Han}
\address{Department of Mathematics,
National University of Singapore, Singapore 119076}
\email{mathanf@nus.edu.sg}

\author{Hao Li}
\address{Department of Mathematics,
National University of Singapore, Singapore 119076}
\email{matlihao@nus.edu.sg}

 \author{Zhi L\"u}
\address{School of Mathematical Sciences, Fudan University, Shanghai, 200433, P.R. China. }
\email{zlu@fudan.edu.cn}




\maketitle

\begin{abstract}
In this paper, we study {\bf twisted Milnor hypersurfaces} and compute their $\hat A$-genus and Atiyah-Singer-Milnor $\alpha$-invariant. Our tool to compute the $\alpha$-invariant is Zhang's analytic Rokhlin congruence formula. We also give some applications about group actions and metrics of  positive scalar curvature on twisted Milnor hypersurfaces.

\end{abstract}

\tableofcontents
\section{Introduction}


\subsection{Twisted Milnor hypersurfaces} Denote by $H_{n_1,n_2}$ the Milnor hypersurface, which is the smooth hypersurface in $\CC P^{n_1}\times\CC P^{n_2}$, the product  of two complex projective spaces,  defined by the equation
\be \label{eqnforMilnor}
x_0y_0+\cdots+x_ky_k=0,\ k=\min\{n_1, n_2\},
\ee
where $[x_0: x_1: \cdots: x_{n_1}]$ and $[y_0: y_1: \cdots: y_{n_2}]$ are the homogeneous coordinates on $\CC P^{n_1}$ and $\CC P^{n_2}$ respectively.
Then $H_{n_1,n_2}$ is Poincar\'e dual to the cohomology class $u+v\in H^2(\CC P^{n_1}\times\CC P^{n_2})$, where  $u$ and $v$ are the generators of $H^*(\CC P^{n_1}, \Z)$ and $H^*(\CC P^{n_2}, \Z)$ respectively.  It is well-known that Milnor hypersurfaces can be used as generators in the unitary bordism ring $\Omega^{U}$ (cf \cite[Section 4.1]{HBJ94}).


In this paper, we consider a generalization of the Milnor hypersurfaces, namely hypersurfaces in certain classes of quasitoric manifolds with polytope being the product of two simplicies $\Delta^{n_1}\times\Delta^{n_2}$ and characteristic matrices being block lower triangular rather than  block diagonal. For details and the background of quasitoric manifold, see Section \ref{toric}.

The quasitoric manifold discussed here can also be regarded as the projective bundle over $\CC P^{n_{1}}$ with fiber $\CC P^{n_{2}}$, i.e.
$$V=\CC P(\eta^{\otimes i_{1}}\oplus\cdots\oplus\eta^{\otimes i_{n_{2}}}\oplus\underline{\CC})\rightarrow \CC P^{n_1},$$
where $\eta$ is the tautological line bundle over $\CC P^{n_{1}}$ and $\underline{\CC}$ is the trivial line bundle. Let $\gamma$ be the vertical tautological line bundle over $V$. Let $u=c_1(\overline{\eta}), v=c_1(\overline{\gamma})\in H^2(V, \Z)$ be the first Chern classes of $\overline{\eta}$ and $\overline{\gamma}$, the complex conjugations of $\eta$ and $\gamma$ respectively. Denote $\mathbf{I}=(i_{1},\cdots ,i_{n_{2}})$.

\begin{definition} Denote by $H_{n_1,n_2}^{\bI}(d_1,d_2)$ a smooth hypersurface in $V$ \poincare\ dual to $d_1u+d_2v$, which we call a {\bf twisted Milnor hypersurface}.
\end{definition}

\begin{remark} {\ }
\begin{enumerate} \item When $\bI=\mathbf{0}=(0, \cdots, 0)$, $H_{n_1,n_2}^{\mathbf{0}}( 1, 1)$ is the classical Milnor hypersurface $H_{n_1, n_2}$.
\vskip.2cm
 \item When $\bI=\mathbf{0}$,\ $n_1\leq n_2$, the smooth hypersurface $H_{n_1,n_2}^{\mathbf{0}}(d_1,1)$ has a model as the zero locus of equation
$$x_0^{d_1}y_0+x_1^{d_1}y_1+\cdots x_{n_1}^{d_1}y_{n_1}=0, $$
which is the generalization of equation (\ref{eqnforMilnor}). However, for general  twisted Milnor hypersurface $H_{n_1,n_2}^{\bI}(d_1,d_2)$, such algebraic geometry models do not exist.

\vskip.2cm

\item $V=\CC P(\eta^{\otimes i_{1}}\oplus\cdots\oplus\eta^{\otimes i_{n_{2}}}\oplus\underline{\CC})$ is a quasitoric manifold over the product $\Delta^{n_1}\times \Delta^{n_2}$ of two simplices. In particular, by the classification of two stage generalized Bott manifolds in \cite{CMS},  $V$ can represent all two stage Bott towers up to diffeomorphism. Furthermore, each $V$ is also a projective toric variety \cite[p. 306]{GK} and when $\bI$ is negative, the complex structures on these projective toric varieties coincide with the natural complex structures on the projective bundles.
\vskip.2cm
 \item When $\bI=(1, 0, \cdots, 0)$, $V$  becomes $L(n_1,n_2)=\CC P(\eta\oplus\underline{\CC}^{n_2})$. It was shown in  \cite{LP} that any generator of the unitary bordism ring $\Omega^{U}$ can be found in $\Z\langle L(n_1,n_2 )\rangle$.
\end{enumerate}
\end{remark}

The main purpose  of this paper is  to study some index theoretical invariants of the twisted Milnor hypersurfaces $H_{n_1,n_2}^{\bI}(d_1,d_2)$. More precisely, we will first pay much attention on the calculations of the $\widehat{A}$-genus and the Atiyah-Milnor-Singer $\alpha$-invariant for $H_{n_1,n_2}^{\bI}(d_1,d_2)$ and then give some applications by applying classical results in geometry.

\subsection{$\widehat{A}$-genus and Atiyah-Milnor-Singer $\alpha$-invariant}
Let $M$ be a $4k$-dimensional closed oriented smooth manifold. Let the formal Chern roots of $TM\otimes \CC$ be $\{\pm  x_j, 1\leq j
\leq 2k\}$. The Hirzebruch $\widehat{A}$-genus is the characteristic number of $M$ defined by
$$ \widehat{A}(M)=\left\langle \prod_{j=1}^{2k} \frac{x_j/2}{\sinh(x_j/2)}, [M] \right\rangle.$$ If the dimension of $M$ is not divisible by 4, define the $\widehat A$-genus of $M$ to be 0. The $\widehat A$-genus
gives a ring homomorphism $\widehat{A}: \Omega_*^{SO}\to \Q.$

Let $M$ be an $n$-dimensional closed smooth spin manifold. The projection map $\pi: M\to pt$ induces an Umkehr homomorphism
$$\pi_!: KO(M) \to KO^{-n}(pt)=KO_n(pt),$$
which is constructed by  using  the  Thom isomorphism for spin bundles in $KO$-theory.
The Atiyah-Milnor-Singer $\alpha$-invariant is defined to be $\alpha(M)=\pi_!(1)$. The $\alpha$-invariant gives a ring homomorphism
$$\alpha: \Omega_*^{spin}\to KO_*(pt)$$
(\cite[\S6, Chapter V]{Bo66}, c.f. \cite[\S16, Chapter III]{LM}).

For spin manifolds, the $\widehat A$-genus and $\alpha$-invariant are  the topological indices of the Atiyah-Singer Dirac operators. By the Bott periodicity, $KO$-theory is 8 periodic. One has
\begin{equation*}
KO_{n}(pt)=\begin{cases}
0, & \mathrm{for} \ n\equiv 3,5,6,7\ \mod\ 8,\\
\Z, &  \mathrm{for}\  n\equiv 0, 4\ \mod\ 8, \\
\Z_2, &  \mathrm{for}\ n\equiv 1, 2\ \mod\ 8,
\end{cases}
\end{equation*}
and
\begin{equation*}
\alpha(M)=\begin{cases}
\widehat{A}(M), & \mathrm{for} \ n\equiv 0\ \mod\ 8,\\
\frac{1}{2}\widehat{A}(M) &  \mathrm{for}\  n\equiv 4\ \mod\ 8.

\end{cases}
\end{equation*}
See  \cite[\S7, Chapter II]{LM} for details.

The $\widehat A$-genus and $\alpha$-invariants have profound applications in geometry. Atiyah and Hirzebruch \cite{AH70} proved that if a compact group acts non-trivially on a compact spin manifold, then the equivariant index of the spin Dirac operator vanishes,
and in particular, the $\widehat A$-genus of the compact manifold vanishes.
Gromov-Lawson \cite{GL} proved that a simply connected closed non-spin manifold of dimension $\geq 5$ always carries a Riemannian metric of positive scalar curvature. In the spin case, it is well known that the $\alpha$-invariant vanishes when the manifolds carry Riemannian metric of positive scalar curvature (due to Lichnerowicz  \cite{Li}  in dimension $4k$ and Hitchin  \cite{Hit} in dimension $8k+1, 8k+2$).
Stolz \cite{Sto}  proved that a simply connected, closed  spin manifold of dimension$\geq 5$ carries a Riemannian metric of positive scalar curvature if and only if the $\alpha$-invariant vanishes.

The $\widehat A$-genus and $\alpha$-invariant have been computed on some classes of manifolds. Brooks \cite{Br83} computed the $\widehat{A}$-genus of complex hypersurfaces and complete intersections in complex projective spaces. Applying his analytic Rokhlin congruence formula established in \cite{Zh93, Zh96R}, Zhang \cite{Zh96E} computed the $\alpha$-invariant of hypersurfaces in complex projective spaces and characterized all the $8k+2$ dimensional hypersurfaces carrying a Riemannian metric of positive scalar curvature. In \cite{FZ}, Feng and  B. H. Zhang  generalized the result in \cite{Zh96E} to $8k+2$ dimensional complete intersections of two hypersurfaces in complex projective spaces. Applying the analytic Rokhlin congruence formula to general complete intersections and using the Selberg-Witten invariant in dimension 4, Fang and Shao \cite{FS} gave the the necessary and sufficient condition for a complete intersection complex projective spaces carrying a Riemannian metric of positive scalar curvature. Recently, Baraglia \cite{Bar} recovered the formula for the $\alpha$-invariant of general complete intersections obtained in \cite{FS} with a different approach.

In this paper, we will compute the $\widehat A$-genus as well as the $\alpha$-invariant of twisted Milnor hypersurfaces and give some applications. The real dimension of twisted Milnor hypersurfaces is always even. We will compute the $\widehat A$-genus without mentioning dimensions (when the dimension is not divisible by 4, the $\widehat A$-genus is automatically 0), and compute the $\alpha$-invariant when the dimension is congruent to 2 mod 8, i.e. when $n_1+n_2\equiv 2\ \mod\, 4$.

In the famous book \cite[\S3.2]{HBJ94}, the $\widehat A$-genus and the elliptic genus of Milnor hypersurfaces have been computed by using the {\bf universal genus}. The universal genus method in \cite{HBJ94} does not work for the computation of the $\widehat A$-genus of twisted Milnor hypersurfaces due to the twistings. We will directly compute it here by combining the characteristic functions and the twisting information. To compute the $\alpha$-invariant, our main tool is Zhang's analytic Rokhlin congruence formula. In the application to the existence of Riemannian metric of positive scalar curvature on twisted Minor hypersurfaces, we express the $\alpha$-invariants by using the dyadic expansion following Zhang \cite{Zh96E}. During the calculation, we discover a very interesting number $A(n,l)$, which is closely related to several classical numbers in number theory. These are summarized in Appendix A.

In a forthcoming paper, we will study the elliptic genus and Witten genus of the twisted Milnor hypersurfaces.

\subsection{Main results}Given $\mathbf{I}=(i_{1},\cdots ,i_{n_{2}})$.
Denote $\sigma_1=\sum_{j=1}^{n_2}i_j.$
Set
\be \label{Fi}
\begin{split}
&F_{n_1,n_2, \bI}(d_1,d_2)\\
=&\sum_{\substack{0\leq r\leq n_2\\ \forall 1\leq j\leq r\\  l_j\geq 1,\sum_{j=1}^{r}l_j\leq n_1,0\leq m_j\leq l_j\\ 1\leq s_1<s_2<\cdots\ s_r \leq n_2}}
(-1)^{\sum_{j=1}^rm_j}\binom{\vec l}{ \vec m}\binom{\frac{d_1+n_1-1+\sigma_1}{2}-\vec s\cdot \vec m}{n_1}\binom{\frac{d_2+n_2-1 }{2}+\sum_{j=1}^rl_j-r}{n_2+\sum_{j=1}^rl_j},
\end{split}
\ee
where  $$\vec l=(l_1, \cdots, l_r), \vec m=(m_1, \cdots, m_r), \vec s=(i_{s_1},\cdots, i_{s_r})$$
and
$$ \binom{\vec l}{\vec m}\coloneqq\binom{l_1}{m_1}\cdots\binom{l_r}{m_r}, \  \vec s \cdot \vec m \coloneqq\sum_{j=1}^ri_{s_j}m_j. $$

\begin{theorem}[ \protect Theorem \ref{Ahat} ] \label{Ahat} For $n_1+n_2\equiv 1\ \mod\ 2$, one has
\be \widehat{A}(H_{n_1,n_2}^{\bI}(d_1,d_2))=F_{n_1,n_2, \bI}(d_1,d_2)-F_{n_1,n_2, \bI}(-d_1 ,-d_2 ).\ee
\end{theorem}

Since $V$ is a complex manifold, it carries a canonical spin$^c$ structure.

\begin{proposition}[ \protect Proposition \ref{spincondition} ] \label{spin}  If  there exists $\  k_{1},k_{2}\in\Z$ such that
$$d_1= 2k_{1}+n_1+1-\sigma_{1} ,d_2= 2k_{2}+n_{2}+1,$$ then $H_{n_1,n_2}^{\bI}(d_1,d_2)$ is a spin manifold carrying the induced spin structure from the spin$^c$ structure of $V$.
\end{proposition}

\noindent{\bf Convention} Throughout this paper, when we say $H_{n_1,n_2}^{\bI}(d_1,d_2)$ is spin, it always mean $H_{n_1,n_2}^{\bI}(d_1,d_2)$ carries the induced spin structure from the spin$^c$ structure of $V$.

\begin{remark}
If $\bI$ is negative and $d_1,d_2$ are positive, then $H_{n_1,n_2}^{\bI}(d_1,d_2)$ is a hyperplane section of $V$.
By the Lefschetz hyperplane theorem (c.f. \cite[Chapter V]{Le}), $i^* : H^2(V, \Z)\rightarrow H^{2}(H_{n_1,n_2}^{\bI}(d_1,d_2), \Z)$ is an isomorphism for $n_1+n_2>3$. Then Proposition \ref{spin} actually gives the necessary and suffcient condition for $H_{n_1,n_2}^{\bI}(d_1,d_2)$ to be spin. Furthermore, the embedding $i$ induces an isomorphism on their fundamental groups. Since $\pi_1(V)=0$, $H_{n_1,n_2}^{\bI}(d_1,d_2)$ is simply connected, and hence the spin structure is unique.
 \end{remark}
\vskip.3cm

\begin{theorem}[ Theorem \ref{13} ] \label{alpha} Assume that $H_{n_1,n_2}^{\bI}(d_1,d_2)$ is spin and $n_1+n_2=4k+2$ (i.e.,  $\dim H_{n_1,n_2}^{\bI}(d_1,d_2) \equiv 2\ \mod\ 8$). Then one has
\be \alpha(H_{n_1,n_2}^{\bI}(d_1,d_2))\equiv F_{n_1,n_2, \bI}(d_1,d_2)\ \mod\ 2. \ee
\end{theorem}

\begin{corollary} [ Example \ref{3.1}, \ref{3.2}] Take $\bI=\mathbf{0}$. For $n_1+n_2\equiv 1\ \mod\ 2$, we have
\be \widehat{A}(H_{n_1,n_2}^{\mathbf{0}}(d_1,d_2))=2\binom{n_1+k_1}{n_1}\binom{n_2+k_2}{n_2}; \ee
and when $H_{n_1,n_2}^{\mathbf{0}}(d_1,d_2)$ is spin and $n_1+n_2\equiv2\ \mod\ 4$, we have
\be \alpha(H_{n_1,n_2}^{\mathbf{0}}(d_1,d_2))\equiv \binom{n_1+k_1}{n_1}\binom{n_2+k_2}{n_2}\ \mod\ 2.\ee
In particular, we have $\widehat{A}(H_{n_1,n_2}^{\mathbf{0}}(1,1))=0$, which coincides with the result in \cite[p. 40]{HBJ94} that the $\widehat{A}$-genus of Milnor hypersurface always vanishes.
\end{corollary}

\begin{corollary}[Example \ref{3.4}]  $H_{2,n_2}^{(j,0,\cdots,0)}(1,n_2+1)$ is spin for $j\equiv 0\ \mod\ 2$, and \\
$$\widehat{A}(H_{2,n_2}^{(j,0,\cdots,0)}(1,n_2+1))\neq0 \iff j\neq 0,-2.$$
\end{corollary}

\begin{remark}
For non-twisted case, one has $\widehat{A}(H_{2,n_2}^{\bf 0}(1,n_2+1))=0$. This provides good examples to illustrate the difference between twisted Milnor hypersurfaces and the usual Milnor hypersurfaces.
Actually, we give a family of twisted Milnor hypersurfaces with non vanishing $\widehat{A}$-genus, see Subsection \ref{examples}.
\end{remark}

\begin{corollary}[Example \ref{3.5}]
For $n_1=1, n_2\equiv 1\ \mod\ 4$,
$$\alpha(H_{n_1,n_2}^{\bI}(d_1,d_2))\equiv \binom{k_1+1}{1}\binom{n_2+k_2}{n_2}+\sigma_1\binom{n_2+k_2}{n_2+1}\ \mod\ 2.$$
\end{corollary}

\begin{corollary}[Example \ref{3.6}]
For $n_1=2, n_2\equiv 0\ \mod\ 4$, denote $\sigma_2=\underset{1\leq j<k\leq n_2}{\sum} i_j\cdot i_k$,
\begin{align*}
 \alpha(H_{n_1,n_2}^{\bI}(d_1,d_2))
\equiv& \binom{k_1+2}{2}\binom{n_2+k_2}{n_2}+\left(-\frac{\sigma_1^2-2\sigma_2}{2}+\frac{(2k_1+3)\sigma_1}{2}\right)\binom{n_2+k_2}{n_2+1}\\
& \ \ +(\sigma_1^2-2\sigma_2)\binom{n_2+k_2+1}{n_2+2}+\sigma_2\binom{n_2+k_2}{n_2+2}\ \mod\ 2.
\end{align*}
\end{corollary}
\vskip.2cm

The Atiyah-Hirzebruch vanishing theorem \cite{AH70} asserts that if the circle $S^1$ acts nontrivially on a connected spin manifold $M$, then $\widehat{A}(M)=0$. We therefore have
\begin{corollary} If $H_{n_1,n_2}^{\bI}(d_1,d_2)$ is spin and $F_{n_1,n_2, \bI}(d_1,d_2)-F_{n_1,n_2, \bI}(-d_1 ,-d_2 )$ does not vanish, then $H_{n_1,n_2}^{\bI}(d_1,d_2)$ does not admit a nontrivial circle action.
\end{corollary}

Based on formula (\ref{Fi}),  we have
\begin{corollary}[ Example \ref{3.3} ] \label{vanishing}
If $n_2$ is even, then $\widehat{A}(H_{n_1,n_2}^{\bI}(d_1,1))=0,\ \forall \ \bI. $
\end{corollary}

In fact, when $\bI=0$, Corollary \ref{vanishing} has a geometric interpretation. Observe that if $n_1\leq n_2$ and $d_2=1$, the smooth hypersurface $H_{n_1,n_2}^{\mathbf{0}}(d_1,1)$ can be described as the zero locus of equation
$$x_0^{d_1}y_0+x_1^{d_1}y_1+\cdots x_{n_1}^{d_1}y_{n_1}=0, $$
where $[x_0: x_1: \cdots: x_{n_1}]$ and $[y_0: y_1: \cdots: y_{n_2}]$ are the homogeneous coordinates on $\CC P^{n_1}$, reps. $\CC P^{n_2}$. There exists a natural circle action on $H_{n_1,n_2}^{\mathbf{0}}(d_1,1)$ defined by
$$\lambda\cdot [x_0: x_1: \cdots: x_{n_1}]=[x_0: \lambda x_1: \cdots: \lambda x_{n_1}] $$
$$\lambda\cdot [y_0: y_1: \cdots: y_{n_2}]=[y_0: \lambda^{-d_1} y_1: \cdots: \lambda^{-d_1} y_{n_2}]$$
where $\lambda\in S^1$.

Naturally we would like to ask
\begin{problem}   Does there exist a non-trivial circle action on $H_{n_1,n_2}^{\bI}(d_1,1)$ for $\bI\neq {\bf 0}$?
\end{problem}


\vskip.3cm

Assume $H_{n_1,n_2}^{\mathbf{I}}(d_1,d_2)$ is spin and $\dim H_{n_1,n_2}^{\mathbf{I}}(d_1,d_2)\geq5$.  By \cite[Theorem A]{KM}, for any oriented manifold $M^{2n}$, any codimension 2 homology class is represented by a submanifold $K\subset M$, and $(M,K)$ is $n$-connected. In the following, let us assume $H_{n_1,n_2}^{\bI}(d_1,d_2)$ to be simply connected. 

Applying Stolz theorem \cite{Sto}, $H_{n_1,n_2}^{\mathbf{I}}(d_1,d_2)$  admits a Riemannian metric of positive scalar curvature (PSC) if and only if  $ \alpha(H_{n_1,n_2}^{\mathbf{I}}(d_1,d_2))=0$.

\begin{corollary}
The spin twisted Milnor hypersurface admits a Riemannian metric of PSC iff 

\noindent $F_{n_1,n_2, \bI}(d_1,d_2)\equiv 0 \ \mod \ 2$. 
\end{corollary}

Motivated by Zhang's results in \cite{Zh96E}, we use the dyadic expansion coefficients to characterize the existence of Riemannian metric of PSC for twisted Milnor hypersurfaces. Let $a_{k}(n)$ be  the coefficient in the dyadic expansion of $n\in\Z$:
$$
n=a_{0}(n)+a_{1}(n)2^{1}+a_{2}(n)2^{2}+\cdots+a_{k}(n)2^{k}, \exists\ k\in \Z.
$$
We are able to give the characterisation for the existence of PSC on two types of twisted Milnor hypersurfaces.
\begin{corollary} [\protect Corollary \ref{n_1=1}]
Assume $n_1=1$,\ $n_1+n_2\equiv 2\ \mod\ 4$, and $k_1=-\frac{n_1+1-d_1-\sigma_1}{2},$\ $k_2=-\frac{n_2+1-d_2}{2}$ are integers. Then the spin $H_{ 1,n_2}^{\mathbf{I}}(d_1,d_2)$ does not admit a Riemannian metric of PSC if and only if one of the following conditions holds,
\begin{itemize}
  \item $k_{2}\geq 0$, $k_{2}\equiv 0\ \mod\ 4$, $k_{1}\equiv 0\ \mod\ 2$, and $\forall\ i, a_{i}([\frac{k_{2}}{4}])+a_{i}([\frac{n_{2}}{4}])\leq 1$;
  \item $k_{2}\geq 0$, $k_{2}\equiv 1\ \mod\ 4$, $\sigma_{1}\equiv 1\ \mod\ 2$ , and $\forall\ i, a_{i}([\frac{k_{2}}{4}])+a_{i}([\frac{n_{2}}{4}])\leq 1$;
  \item $k_{2}\geq 0$, $k_{2}\equiv 2\ \mod\ 4$, $k_{1}+\sigma_{1}\equiv 0\ \mod\ 2$, and $\forall\ i, a_{i}([\frac{k_{2}}{4}])+a_{i}([\frac{n_{2}}{4}])\leq 1$;
  \item $k_{2}\leq -n_{2}-1$, $-k_{2}\equiv 0\ \mod\ 4$, $k_{1}\equiv 0\ \mod\ 2$, and $\forall\ i,a_{i}([\frac{-k_{2}-1-n_{2}}{4}])+a_{i}([\frac{n_{2} }{4}])\leq 1$;
  \item $k_{2}\leq -n_{2}-1$, $-k_{2}\equiv 2\ \mod\ 4$, $k_{1}+\sigma_{1}\equiv 0\ \mod\ 2$, $\forall\ i, a_{i}([\frac{-k_{2}-1-n_{2}}{4}])+a_{i}([\frac{n_{2} }{4}])\leq 1$;
  \item $k_{2}\leq -n_{2}-1$, $-k_{2}\equiv 3\ \mod\ 4$, $\sigma_{1}\equiv 1\ \mod\ 2,$ and $\forall\ i, a_{i}([\frac{-k_{2}-1-n_{2}}{4}])+a_{i}([\frac{n_{2} }{4}])\leq 1$.
\end{itemize}
\end{corollary}

\begin{corollary}  [\protect Corollary \ref{n_1=2}]
Assume  $n_1=2$,\ $n_1+n_2\equiv 2\ \mod\ 4$,  and $k_1=-\frac{n_1+1-d_1-\sigma_1}{2},$ \ $k_2=-\frac{n_2+1-d_2}{2}$ are integers. Then the spin $H_{2,n_2}^{\mathbf{I}}(d_1,d_2)$ does not admit a Riemannian metric of PSC if and only if one of the following condition holds
\begin{itemize}
  \item $k_{2}\geq 0$, $k_{2}\equiv 0\ \mod\ 4$, $k_{1}\equiv 0\ \textrm{or}\ 1\ \mod\ 4$, $ \forall\ i, a_{i}([\frac{k_{2} }{4}])+a_{i}(\frac{n_{2}}{4})\leq 1$;
  \item $k_{2}\geq 0$, $k_{2}\equiv 1\ \mod\ 4$, $ \binom{k_1+2}{2}  +\frac{\sigma_1^2-2\sigma_2}{2}+\frac{(2k_1+3)\sigma_1}{2}  \equiv 1\  \mod\ 2,\forall\ i, a_{i}([\frac{k_{2} }{4}])+a_{i}(\frac{n_{2}}{4})\leq 1;$
  \item  $k_{2}\geq 0$, $k_{2}\equiv 2\ \mod\ 4$, $  \binom{k_1+2}{2} +\sigma_1^2-\sigma_2 \equiv 1\  \mod\ 2,\forall\ i, a_{i}([\frac{k_{2} }{4}])+a_{i}(\frac{n_{2}}{4})\leq 1;$
  \item $ k_{2}\geq 0$, $k_{2}\equiv 3\ \mod\ 4$, $ \binom{ k_{1}+2} {2}+\frac{\sigma_{1}(2k_{1}+3-\sigma_{1})}{2} \equiv 1\  \mod\ 2,\forall\ i, a_{i}([\frac{k_{2} }{4}])+a_{i}(\frac{n_{2}}{4})\leq 1;$
   \item $k_{2}= -n_{2}-1$, $ \frac{(k_{1}+1)(k_{1}+2)}{2}+\frac{\sigma_{1}(\sigma_{1}-2k_{1}-3)}{2} \equiv 1\  \mod\ 2;$

  \item $k_{2}\leq -n_{2}-2$, $-k_{2}\equiv 0\ \mod\ 4$, $ k_{1}\equiv 0\ \mathrm{or}\ 1\ \mod\ 4,\forall\ i, a_{i}([\frac{-k_{2}-1-n_{2}}{4}])+a_{i}(\frac{n_{2}}{4})\leq 1;$
  \item $k_{2}\leq -n_{2}-2$, $-k_{2}\equiv 1\ \mod\ 4$, $ \binom{k_1+2}{2}+\frac{\sigma_{1}(\sigma_{1}-2k_{1}-3)}{2} \equiv  1\ \mod\ 2,\forall\ i, a_{i}([\frac{-k_{2}-1-n_{2}}{4}])+a_{i}(\frac{n_{2}}{4})\leq 1;
 $
  \item  $k_{2}\leq -n_{2}-2$, $-k_{2}\equiv 2\ \mod\ 4$, $ \binom{k_1+2}{2}+\sigma_1^2-\sigma_2 \equiv  1\ \mod\ 2,\forall \ i,a_{i}([\frac{-k_{2}-1-n_{2}}{4}])+a_{i}(\frac{n_{2}}{4})\leq 1;$
  \item$k_{2}\leq -n_{2}-2$, $-k_{2}\equiv 3\ \mod\ 4$, $ \binom{k_1+2}{2}+\frac{(2k_1+3+3\sigma_1)\sigma_1}{2}-\sigma_2 \equiv  1\ \mod\ 2,\forall\ i, a_{i}([\frac{-k_{2}-1-n_{2}}{4}])\\+a_{i}(\frac{n_{2}}{4})\leq 1;$
\end{itemize}
\end{corollary}

$\, $

This paper is organized as follows. In Section 2, we give some topological preliminaries on twisted Milnor hypersurfaces.
In Section 3, we give explicit formulas for $\widehat{A}$-genus and $\alpha$ invariant of $H_{n_1,n_2}^{\bI}(d_1,d_2)$ by $F_{n_1,n_2,\bI}(d_1,d_2)$. During the computation, we use the binomial number $A(n,l)$ and its properties.
In Section 4, we give some applications of these two invariants to twisted Milnor hypersurfaces.

$\, $

\noindent {\bf Acknowledgement} Fei Han would like to thank Prof. Kefeng Liu and Prof. Weiping Zhang for helpful discussions. 
Jingfang Lian  and Zhi L\"u are partially supported by the grant from NSFC (No. 11971112); Fei Han is partially supported by the grant AcRF R-146- 000-218-112 from National University of Singapore; Hao Li is partially supported by the grant R-146-000-322-114 from National University of Singapore.

\vskip .5cm

\section{Topological preliminaries on twisted Milnor hypersurfaces}\label{toric}

Let $P^n$ be a  simple convex polytope of dimension $n$. A {\bf quasitoric manifold} \cite{DJ} $M^{2n}$ over $P^n$ is a smooth $T^n$-manifold $M^{2n}$  satisfying the following two conditions :
\begin{enumerate}
  \item the action is locally standard;
  \item there is continuous projection $\pi: M^{2n}\rightarrow P^n$ whose fibers are $T^n$-orbits.
\end{enumerate}
Note that $(2)$ says that the orbit space of $T^n$ on $M^{2n}$ is homeomorphic to $P^n$.

Denote by $\mathcal{F}$ the set of  codimension one faces of $P^n$. For every $F\in \mathcal{F}, x\in \mathrm{Int}(F)$. The isotropy group of $x$ is a codimension one subgroup of $T^n$; this isotropy subgroup is determined by a primitive vector $v\in \Z^{n}$; thus define a {\bf characteristic function }$\lambda:\mathcal{F}\rightarrow \Z^{n}$, and we call its corresponding matrix the  {\bf characteristic matrix} (detail can be found in \cite{DJ}.)

For example, identify $T^{n}$ with $T^{n+1}/\langle(g,\cdots,g),g\in T^n\rangle$ and $T^{n}$ acts on $\CC P^n$ in the usual manner, the quotient space is then the simplex $\Delta^n$, where $$\Delta^n=\{(x_1,\cdots,x_n)\in \R^n| x_i \geq 0,\ i=1,2,\cdots,n, \    x_1+\cdots+x_n \leq 1   \}.$$
Then $\CC P^n$
can be viewed a quasitoric manifold over the simplex $\Delta^n$ with characteristic matrix:
\begin{gather*}
\begin{pmatrix}
1     & 0 &\cdots & 0 & -1      \\
0     & 1 &\cdots &   & -1        \\
\vdots&   &\ddots &   & \vdots   \\
0     &   &       & 1 & -1      \\
\end{pmatrix}
\end{gather*}

Now let's consider more complicated combinatorics: the product of two simplices $\Delta^{n_1}\times\Delta^{n_2}$. For example, $\Delta^1\times\Delta^2$ looks like

\begin{center}
\begin{tikzpicture}
\draw (-0.5,0)--(1.5,0);

\draw  (2.7,0.3)--(3.3,-0.3);
\draw  (2.7,-0.3)--(3.3,0.3);

\draw (4,-1.5)--(7,-1.5);
\draw (4,-1.5)--(5.5,1.5);
\draw (7,-1.5)--(5.5,1.5);

\draw (7.6,0.1)--(8,0.1);
\draw (7.6,-0.1)--(8,-0.1);

\draw(9,1)--(12,1);
\draw (9,-1.5)--(12,-1.5);
\draw (9,1)--(9,-1.5);
\draw (12,1)--(12,-1.5);
\draw(9,1)--(11,1.5);
\draw (11,1.5)--(12,1);
\draw[dashed] (11,1.5)--(11,-1);
\draw[dashed] (9,-1.5)--(11,-1);
\draw[dashed] (12,-1.5)--(11,-1);
\end{tikzpicture}
\end{center}
\vskip.3cm
The product of two projective spacess $\CC P^{n_1}\times\CC P^{n_2}$ is a quasitoric manifold over the polytope $\Delta^{n_1}\times\Delta^{n_2}$ with the block diagonal characteristic matrix:
$\, $

\begin{gather*}
\begin{pmatrix}
\cooloverbrace{n_1}{1& 0 &\cdots & 0} & -1 &   &        &   &     \\
0     & 1 &\cdots &   & -1      &   &        &   &    \\
\vdots&   &\ddots &   & \vdots  &   &        &   &   \\
0     &   &       & 1 & -1      &   &        &   &   \\
      &   &       &   &      & 1 &        &   & -1 \\
      &   &       &   &      &   &\ddots  &   & \cdots \\
      &   &       &   &     &     &        & 1 & -1\\
\end{pmatrix}
\begin{matrix}
    \\
    \\
    \\
    \\
    \coolrightbrace{-1 \\ \cdots \\ -1}{n_2}
  \end{matrix}
\end{gather*}

Twisting the block diagonal characteristic matrix to be a block lower triangular characteristic matrices can give interesting new quasitoric manifolds. Let $V$ be a quasitoric manifold, whose corresponding  polytope is $ \Delta^{n_1}\times\Delta^{n_2}$ while the characteristic matrix is

\begin{gather*}
\begin{pmatrix}
\cooloverbrace{n_1}{1& 0 &\cdots & 0 }& -1 &   &        &   &     \\
0     & 1 &\cdots &   & -1      &   &        &   &   \\
\vdots&   &\ddots &   & \vdots  &   &        &   & \\
0     &   &       & 1 & -1      &   &        &   &   \\
      &   &       &   & i_1     & 1 &        &   & -1\\
      &   &       &   & \vdots  &   &\ddots  &   & \cdots \\
      &   &       &   & i_{n_2} &   &        & 1 & -1 \\
\end{pmatrix}
\begin{matrix}
    \\
    \\
    \\
    \\
    \coolrightbrace{-1 \\ \cdots \\ -1}{n_2}
  \end{matrix}
\end{gather*}

$V$ can also be considered as the total space of the projectivisation $$\CC P(\eta^{\otimes i_{1}}\oplus\cdots\oplus\eta^{\otimes i_{n_{2}}}\oplus\underline{\CC})\rightarrow \CC P^{n_1}$$( c.f. \cite[Section 7.8]{BP}), where $\eta$ is the tautological line bundle over $\CC P^{n_{1}}$. Denote by $\overline{\eta}$ the conjugate bundle of $\eta$. Let $\gamma$ be the tautological vertical line bundle over $V$. The manifold $V$ is the total space of  a bundle over $\CC P^{n_{1}}$ with fiber $\CC P^{n_{2}}$.

The  tangent bundle  and cohomology ring structure of $V$ are clear from the following theorem.

\begin{theorem}[\protect Borel and Hirzebruch \cite{BH58}]
Let $p:\CC P(\xi)\rightarrow X$ be the projectivization of a complex n-plane bundle $\xi$ over a complex manifold X, and  $\gamma$  the tautological vertical line bundle over $\CC P(\xi)$. Then there is an isomorphism of vector bundles
$$
T \CC P(\xi)\oplus\underline{\CC}\cong p^{*}T X\oplus(\overline{\gamma}\otimes p^{*}\xi),
$$
where $\underline{\CC}$ denotes a trivial line bundle over $\CC P(\xi)$. Futhermore,
$$
H^{*}(\CC P(\xi);\Z)\cong H^{*}(X)[c_1(\overline{\gamma})]/\langle c_{n}(\overline{\gamma}\otimes p^{*}\xi)\rangle.
$$

\end{theorem}

\indent

In this paper, $X=\CC P^{n_{1}}, \xi=\eta^{\otimes i_{1}}\oplus\cdots\oplus\eta^{\otimes i_{n_{2}}}\oplus\underline{\CC}$. Let $u=p^{*}c_{1}(\overline{\eta}),v=c_{1}(\overline{\gamma})$. By the above theorem, we have
$$
T V\oplus \underline{\CC}\cong p^{*}T \CC P^{n_{1}}\oplus\overline{\gamma}\otimes p^{*}(\eta^{\otimes i_{1}}\oplus\cdots\oplus\eta^{\otimes i_{n_{2}}}\oplus\underline{\CC}),
$$
 $$
 H^{*}(V)\cong H^{*}(\CC P^{n_{1}};\Z)[v]/c_{n_2+1}(\overline{\gamma}\otimes p^{*}(\eta^{\otimes i_{1}}\oplus\cdots\oplus\eta^{\otimes i_{n_{2}}}\oplus\underline{\CC})).
 $$
 Furthermore
 $$\overline{\gamma}\otimes  p^{*}(\eta^{\otimes i_{1}}\oplus\cdots\oplus\eta^{\otimes i_{n_{2}}}\oplus\underline{\CC})=(\overline{\gamma}\otimes p^{*}\eta^{\otimes i_{1}})\oplus\cdots\oplus(\overline{\gamma}\otimes  p^{*}\eta^{\otimes i_{n_{2}}})\oplus(\overline{\gamma}\otimes\underline{\CC}).$$
 and
 $$c(\overline{\gamma}\otimes\eta^{\otimes i_{1}})=1+c_{1}(\overline{\gamma})+c_{1}(\eta^{\otimes i_{1}})=1+v-i_1u.$$
 Thus we have
 $$
 c(V)=(1+u)^{n_{1}+1}(1+v)\prod_{j=1}^{n_2}(1+v-i_{j}u).
 $$
 Therefore $$c_1(V)=(n_1+1-\sigma_1)u+(n_2+1)v$$ and
$$p(V)=(1+u^2)^{n_{1}+1}(1+(v-i_{1}u)^2)\cdots(1+(v-i_{n_{2}}u)^2)(1+v^2).
$$

We also have $H^{*}(\CC P^{n_{1}};\Z)\cong \Z[u]/\langle u^{n_{1}+1}\rangle$ and therefore
$$H^{*}(V)\cong \Z[u,v]/\langle u^{n_{1}+1},v(v-i_{1}u)\cdots (v-i_{n_{2}}u)\rangle$$
\vskip.2cm
The following result should be known to experts, although we did't find it in the literature. We state it here and give the proof.
\begin{lemma}
  $$\langle u^{n_{1}}v^{n_{2}},[V]\rangle=1.$$
 \end{lemma}
 \begin{proof}
 We have fibration $\CC P^{n_{2}}\rightarrow V\rightarrow \CC P^{n_{1}}$. By the fibration property of Euler characteristic \cite[p. 481]{SpE},
we have $\chi(V)=\chi(\CC P^{n_{1}})\chi(\CC P^{n_{2}})$.

 And  since $H^{*}(\CC P^{n},\Z)=\Z[u]/\langle u^{n+1}\rangle$ and $\deg(u)=2$, so the odd dimension of $H^{*}(\CC P^{n},\Z)$  vanishes.
 So $$\chi(\CC P^{n})=\sum_{i=0}^{n}(-1)^{i}\dim(H^{i}(\CC P^{n},\Z))=n+1.$$
 Thus$$\chi(V)=(n_{1}+1)(n_{2}+1),$$
since $\cl{c_{n_1+n_2}(V),[V]}=\chi(V)$. And
 $$c_{n_1+n_2}(V)=(n_{1}+1)(n_{2}+1)u^{n_{1}}v^{n_{2}}$$ since  $u^{n_{1}+1}=0$.

Therefore
$$ \cl{(n_{1}+1)u^{n_{1}}(n_{2}+1)v^{n_{2}},[V]}=(n_{1}+1)(n_{2}+1)$$
and we have $\langle u^{n_{1}}v^{n_{2}},[V]\rangle=1$.
 \end{proof}
Since $H^{2(n_1+n_2)}(V)\cong \Z[u^{n_{1}}v^{n_{2}}],$ assume $u^{n_1-k}v^{n_{2}+k}=\beta_{k}u^{n_1}v^{n_{2}}$ for $0\leq k\leq n_1$.

 \begin{lemma}\label{beta}
 $$\beta_{k}=\sum\limits_{\substack{\forall 1\leq i\leq n_2,\ p_i\geq 0\\ p_1+p_2+\cdots+p_{n_{2}}=k}} i_{1}^{p_1}\cdots i_{n_{2}}^{p_{n_2}}.$$
 \end{lemma}
\begin{proof}
If $k=0$, obviously $\beta_{0}=1$.

If $k>0$, since $ v(v-i_{1}u)\cdots (v-i_{n_{2}}u)=0$, we have
\begin{align*}
0&=v^k(v-i_{1}u)\cdots (v-i_{n_{2}}u)\\
&=v^{n_{2}+k}-\sigma_1 uv^{n_{2}+k-1}+\cdots+(-1)^k\sigma_k u^{k}v^{n_{2}}\\
&=(\beta_k-\sigma_1 \beta_{k-1}+\cdots+(-1)^k\sigma_k \beta_0)u^{k}v^{n_{2}}.
\end{align*}

Thus $\beta_k$ is the unique solution of the equation,
$$\beta_k-\sigma_1 \beta_{k-1}+\cdots+(-1)^k\sigma_k \beta_0=0.$$

On the other hand, since $u^{n_1+1}=0,$ we have

\begin{align*}
1&=(1-i_{1}u)\cdots (1-i_{n_{2}}u)\frac{1}{(1-i_{1}u)}\cdots \frac{1}{(1-i_{n_{2}})}\\
&=(\sum_{i=0}^{n_{2}}(-1)^i \sigma_i u^i)\prod_{j=1}^{n_{2}}(1+i_j u+\cdots+(i_j)^{n_1} u^{n_1})\\
&=(\sum_{i=0}^{n_{2}}(-1)^i \sigma_i u^i)\sum_{m=0}^{n_1}(\sum_{ p_1+\cdots+p_{n_2}=m} i_{1}^{p_1}\cdots i_{n_{2}}^{p_{n_2}})u^m\\
&=\sum_{k=0}^{n_1}\sum_{m+i=k}((-1)^i \sigma_i \beta'_m )u^{m+i}\\
&=\sum_{k=0}^{n_1}(\beta'_k+(-1)\sigma_1\beta'_{k-1}+\cdots+(-1)^k\sigma_k\beta'_0)u^k,
\end{align*}

where $\beta'_k=\!\sum\limits_{\substack{\forall 1\leq i\leq n_2,\ p_i\geq 0\\ p_1+p_2+\cdots+p_{n_{2}}=k}}\! i_{1}^{p_1}\cdots i_{n_{2}}^{p_{n_2}}$.
Thus $\beta'_k=\beta_k$.

\end{proof}

 \begin{remark}
 In this paper, we will mainly use the expression in Lemma \ref{beta}. One can also give an expression of $\beta_k$ by the elementary symmetry polynomials of $\{i_1,\cdots,i_{n_2}\}$,
 $$
 \beta_{k}=\!\!\sum_{p_{1}+2p_{2}+\cdots+n_{2}p_{n_{2}}=k}\!\!(-1)^{k-p}\sigma_{1}^{p_{1}}\sigma_{2}^{p_{2}}\cdots\sigma_{n_{2}}^{p_{n_{2}}}\frac{p!}{p_{1}!\cdots p_{n_{2}}!} ,$$
 where $p=p_{1}+\cdots+p_{n_{2}}$ and $\sigma_{i}=\sigma_{i}(i_{1},\cdots,i_{n_{2}})$ is the $i$-th elementary symmetry polynomial.
 \end{remark}

\begin{corollary}
$$\langle u^{n_{1}-k}v^{n_{2}+k},[V]\rangle=\beta_{k}=\sum\limits_{\substack{\forall 1\leq i\leq n_2,\ p_i\geq 0\\ p_1+p_2+\cdots+p_{n_{2}}=k}} i_{1}^{p_1}\cdots i_{n_{2}}^{p_{n_2}}.$$
\end{corollary}
\begin{proposition} \label{spincondition} If  there exists $\  k_{1},k_{2}\in\Z$ such that
$$d_1= 2k_{1}+n_1+1-\sigma_{1} ,d_2= 2k_{2}+n_{2}+1,$$ where $\sigma_1=\sum_{j=1}^{n_2}i_j,$ then $H_{n_1,n_2}^{\bI}(d_1,d_2)$ is a spin manifold carrying the induced spin structure (c.f. \cite{Zh96R}).
\end{proposition}
\begin{proof}
Since $V$ is a complex manifold, it carries a canonical spin$^c$ structure \cite{KT}.

As $H_{n_1,n_2}^{\bI}(d_1,d_2)$ is \poincare\ dual to $d_1u+d_2v$, by adjunction formula, there exists a complex line bundle $\xi$ over $V$ with $c_{1}(\xi)=d_1u+d_2v$,
such that the normal bundle $\nu$ of the inclusion $H_{n_1,n_2}^{\bI}(d_1,d_2)\overset{i}{\hookrightarrow} V$ is the pullback of $\xi$.

Therefore we have
 $$
 TV|_{H_{n_1,n_2}^{\bI}(d_1,d_2)}\cong TH_{n_1,n_2}^{\bI}(d_1,d_2)\oplus \nu\cong TH_{n_1,n_2}^{\bI}(d_1,d_2)\oplus i^{*}(\xi)
 $$
 and
 $$
 i^{*}c(V)=c(H_{n_1,n_2}^{\bI}(d_1,d_2))i^{*}c(\nu)
 $$
 $$
 \Rightarrow c(H_{n_1,n_2}^{\bI}(d_1,d_2))=i^{*}\{c(V)(1+d_1u+d_2v)^{-1}\}
 $$
 $$
 \Rightarrow
 c_{1}(H_{n_1,n_2}^{\bI}(d_1,d_2))=(n_1+1-\sigma_1-d_1)i^{*}u+(n_2+1-d_2)i^{*}v.
 $$

 Thus $\omega_2\equiv c_1 \equiv 0 \ \mod \ 2$, $H_{n_1,n_2}^{\bI}(d_1,d_2)$ is spin.

\end{proof}
\begin{remark}
If $\bI$ is negative and $d_1,d_2$ are positive, then $H_{n_1,n_2}^{\bI}(d_1,d_2)$ is a hyperplane section of $V$.
By the Lefschetz hyperplane theorem (c.f. \cite{Le}), $i^* : H^2(V, \Z)\rightarrow H^{2}(H_{n_1,n_2}^{\bI}(d_1,d_2), \Z)$ is an isomorphism for $n_1+n_2>3$. Then the above proposition actually gives the necessary and suffcient condition for $H_{n_1,n_2}^{\bI}(d_1,d_2)$ to be spin. Furthermore, the embedding $i$ induces an isomorphism on their fundamental groups. Since $\pi_1(V)=0$, \ $H_{n_1,n_2}^{\bI}(d_1,d_2)$ is simply connected, the spin structure is unique.
 \end{remark}

\section{$\hat{A}$-genus and $\alpha$-invariant of twisted Milnor hypersurfaces }

In this section, we compute the  $\widehat{A}$-genus and $\alpha$-invariant of $H_{n_1,n_2}^{\bI}(d_1,d_2)$. The computation involves two combinatoric numbers $A(n,l)$ and $F_{n_1,n_2,\bI}(d_1,d_2)$, which we will deal with in Subsections \ref{A} and \ref{F} first.

\subsection{The number $A(n,l)$}\label{A}{\ }\\
\indent

Define
\begin{equation*}
A(n,l)=
\begin{cases}
\frac{1}{n!}\sum\limits_{m=0}^{l}(-1)^{l-m}\binom{l }{m}m^{n }, &  0\leq l\leq n;\\
0,&  \text{otherwise.}
\end{cases}
\end{equation*}

Denote by $$T(x)=\frac{x}{1-e^{-x}}=\sum_{m=0}^{\infty}\frac{B_{m}(1)}{m!}x^{m},$$
 the Todd series (c.f. \cite{CG}). Let $T^{(n)}(x)$ be the $n$-th derivative of $T(x)$.
\begin{lemma}\label{deriviate of g} The following identity holds,
$$\sum_{m=0}^{n}\frac{(-v)^{m}}{m!}T^{(m)}(v)=(T(v)-vT'(v))\left\{\sum_{l=1}^n A(n,l)(-v)^{n-l}T(v)^{l-1}\right\}.$$
\end{lemma}
\begin{proof}
When $n=1$, $$T(v)-vT'(v)=(T(v)-vT'(v))\{A(1,1)(-v)^0\}$$
holds since $A(1,1)=1$.
\vskip.2cm
When $n=2$, since $T(v)=\frac{ve^v}{e^v-1}$, we have

$$T(v)-vT'(v)+\frac{v^2}{2}T''(v)=(T(v)-vT'(v))\{\frac{-v}{2}+T(v)\}$$
Since $A(2,1)=\frac{1}{2}$ and $ A(2,2)=1$, the lemma holds.
\vskip.2cm
Assume the lemma holds for any integer $\leq n$. Multiply $v$ to both sides of
$$\sum_{m=0}^{n}\frac{(-v)^{m}}{m!}T^{(m)}(v)=(T(v)-vT'(v))\left\{\sum_{l=1}^n A(n,l)(-v)^{n-l}T(v)^{l-1}\right\},$$
we have
$$-\sum\limits_{m=0}^{n}\frac{(-v)^{m+1}}{m!}T^{(m)}(v)=v(T(v)-vT'(v))\left\{\sum\limits_{l=1}^n A(n,l)(-v)^{n-l}T(v)^{l-1}\right\}.$$

Differentiating the left hand side gives
\begin{align*}
& \frac{\mathrm{d}}{\mathrm{d}v}\{vT(v)-v^2T'(v)+\cdots+(-1)^{n-1}\frac{v^{n}}{(n-1)!}T^{(n-1)}(v)+(-1)^n\frac{v^{n+1}}{n!}T^{(n)}(v)\}\\
=&(T(v)+vT'(v))-(2vT'(v)+v^2T''(v))+\cdots+(-1)^{n-1}\{\frac{nv^{n-1}}{(n-1)!}T^{(n-1)}(v)\\
&+\frac{v^{n}}{(n-1)!}T^{(n)}(v)\}+\displaystyle(-1)^n\left\{\frac{(n+1)v^n}{n!}T^{(n)}(v)+\frac{v^{n+1}}{n!}T^{(n+1)}(v)\right\}\\
=&T(v)-(-1+2)vT'(v)+\cdots+(-1)^n(-1+\frac{n+1}{n})\frac{v^n}{(n-1)!}T^{(n)}(v)+(-1)^{n}\frac{v^{n+1}}{n!}T^{(n+1)}(v)\\
=&T(v)-vT'(v)+\cdots+(-1)^n\frac{v^n}{n!}T^{(n)}(v)+(-1)^{n}\frac{v^{n+1}}{n!}T^{(n+1)}(v)\\
=&(T(v)-vT'(v))\left\{\sum\limits_{l=1}^n A(n,l)(-v)^{n-l}T(v)^{l-1}\right\}-\frac{(-v)^{n+1}}{n!}T^{(n+1)}(v),
\end{align*}
which equals to the differentiation of the right hand side:
\begin{small}
\begin{align*}
&(T(v)-vT'(v))\left\{\sum\limits_{l=1}^n A(n,l)(-v)^{n-l}T(v)^{l-1}\right\}+v(-vT''(v))\left\{\sum\limits_{l=1}^n A(n,l)(-v)^{n-l}T(v)^{l-1}\right\}\\
&+v(T(v)-vT'(v))\frac{\mathrm{d}}{\mathrm{d}v}\left\{\sum\limits_{l=1}^n A(n,l)(-v)^{n-l}T(v)^{l-1}\right\}.
\end{align*}\end{small}

So we get
\begin{small}
\begin{align*}
&\frac{(-v)^{n+1}}{n!}T^{(n+1)}(v)\\
=&v(vT''(v))\left\{\sum\limits_{l=1}^n A(n,l)(-v)^{n-l}T(v)^{l-1}\right\}-v(T(v)-vT'(v))\frac{\mathrm{d}}{\mathrm{d}v}\left\{\sum\limits_{l=1}^n A(n,l)(-v)^{n-l}T(v)^{l-1}\right\}.
\end{align*}\end{small}

And since $v^2T''(v)=(T(v)-vT'(v))(2T(v)- v-2)$, we have
\begin{small}
\begin{align*}
&\frac{(-v)^{n+1}}{(n+1)!}T^{(n+1)}(v)\\
=&\frac{v^2T''(v)}{n+1}\left\{\sum\limits_{l=1}^n A(n,l)(-v)^{n-l}T(v)^{l-1}\right\} -\frac{v(T(v)-vT'(v))}{n+1}\cdot\frac{\mathrm{d}}{\mathrm{d}v}\left\{\sum\limits_{l=1}^n A(n,l)(-v)^{n-l}T(v)^{l-1}\right\}\\
=& \frac{ (T(v)-vT'(v))(2T(v)- v-2)}{n+1}\left\{\sum\limits_{l=1}^n A(n,l)(-v)^{n-l}T(v)^{l-1}\right\}\\
&-\frac{v(T(v)-vT'(v))}{n+1}\cdot\frac{\mathrm{d}}{\mathrm{d}v}\left\{\sum\limits_{l=1}^n A(n,l)(-v)^{n-l}T(v)^{l-1}\right\}\\
=&\frac{T(v)-vT'(v)}{n+1}\left\{(2T(v)-v-2)\sum\limits_{l=1}^n A(n,l)(-v)^{n-l}T(v)^{l-1}-v\frac{\mathrm{d}}{\mathrm{d}v}[ \sum\limits_{l=1}^n A(n,l)(-v)^{n-l}T(v)^{l-1}]\right\}.
\end{align*}\end{small}
Thus
\begin{small}
\begin{align*}
&T(v)-vT'(v)+\cdots+\frac{(-v)^n}{n!}T^{(n)}(v)+\frac{(-v)^{n+1}}{(n+1)!}T^{(n+1)}(v)\\
=&(T(v)-vT'(v))\{\sum\limits_{l=1}^n A(n,l)(-v)^{n-l}T(v)^{l-1}\}+\\
&\frac{T(v)-vT'(v)}{n+1}\cdot\left\{(2T(v)-v-2)\sum\limits_{l=1}^n A(n,l)(-v)^{n-l}T(v)^{l-1}-v\frac{\mathrm{d}}{\mathrm{d}v}[ \sum\limits_{l=1}^n A(n,l)(-v)^{n-l}T(v)^{l-1}]\right\}\\
=&\frac{T(v)-vT'(v)}{n+1}\cdot\{(2T(v)-v+n-1)\sum\limits_{l=1}^n A(n,l)(-v)^{n-l}T(v)^{l-1}\\
&-v\cdot\frac{\mathrm{d}}{\mathrm{d}v}[ \sum\limits_{l=1}^n A(n,l)(-v)^{n-l}T(v)^{l-1}]\}\\
=&\frac{T(v)-vT'(v)}{n+1}\cdot\{(2T(v)-v+(n-1))\sum\limits_{l=1}^n A(l,n)(-v)^{n-l}T(v)^{l-1}\\
&+ v\sum\limits_{l=1}^n  A(n,l)\cdot(n-l) (-v)^{n-l-1}T(v)^{l-1}-v\sum\limits_{l=1}^n A(n,l)(-v)^{n-l}(l-1)T(v)^{l-2}T'(v)\}\\
=&\frac{T(v)-vT'(v)}{n+1}\{(2T(v)-v+(n-1))\sum\limits_{l=1}^n A(n,l)(-v)^{n-l}T(v)^{l-1}\\
&+\sum\limits_{l=1}^n(l-n) A(n,l)\displaystyle\cdot(-v)^{n-l}T(v)^{l-1}-v\sum\limits_{l=1}^n A(n,l)(-v)^{n-l}(l-1)T(v)^{l-2}T'(v)\}\\
=&\frac{T(v)-vT'(v)}{n+1}  \cdot\{\sum\limits_{l=1}^n (2T(v)-v+l-1)A(n,l)(-v)^{n-l}T(v)^{l-1}\\
&-vT'(v)\sum\limits_{l=1}^n  A(n,l)\cdot(-v)^{n-l}(l-1)T(v)^{l-2}\}.
\end{align*}
\end{small}
For the first part in the bracket, we have
\begin{small}
\begin{align*}
&\sum\limits_{l=1}^n(2T(v)-v+l-1) A(n,l)(-v)^{n-l}T(v)^{l-1}\\
=&2\sum\limits_{l=1}^n A(n,l)(-v)^{n-l}T(v)^{l}+\sum\limits_{l=1}^n A(n,l)(-v)^{n-l+1}T(v)^{l-1}+\sum\limits_{l=1}^n (l-1)A(n,l)(-v)^{n-l}T(v)^{l-1}\\
=&\sum\limits_{l=2}^{n+1}2A(n,l-1 )(-v)^{n-l+1}T(v)^{l-1}+\sum\limits_{l=1}^n A(n,l)(-v)^{n-l+1}T(v)^{l-1} +\sum\limits_{l=1}^n (l-1)A(n,l)(-v)^{n-l}T(v)^{l-1}\\
=&\sum\limits_{l=1}^{n+1} (A(n,l)+2A(n,l-1 ))(-v)^{n-l+1}T(v)^{l-1}+\sum\limits_{l=1}^n (l-1)A(n,l)(-v)^{n-l}T(v)^{l-1}.
\end{align*}\end{small}

Since $vT'(v)=T(v)(v-T(v)+1)$, for the second part in the bracket, we have
\begin{small}
\begin{align*}
&-vT'(v)\sum\limits_{l=1}^n A(n,l)(-v)^{n-l}(l-1)T(v)^{l-2}\\
=&(-v+T(v)-1)\sum\limits_{l=1}^n (l-1)A(n,l)(-v)^{n-l}T(v)^{l-1}\\
=&\sum\limits_{l=1}^n (l-1)A(n,l)(-v)^{n-l+1}T(v)^{l-1}+\sum\limits_{l=1}^n (l-1)A(n,l)(-v)^{n-l}T(v)^{l}\\
&+\sum\limits_{l=1}^n(-1)(l-1)A(n,l)(-v)^{n-l}T(v)^{l-1}\\
=&\sum\limits_{l=1}^n (l-1)A(n,l)(-v)^{n-l+1}T(v)^{l-1}+\sum\limits_{l=2}^{n+1} (l-2)A(n,l-1 )(-v)^{n-l+1}T(v)^{l-1}\\
&+\sum\limits_{l=1}^n(-1)(l-1)A(n,l)(-v)^{n-l}T(v)^{l-1}\\
=&\sum\limits_{l=1}^{n+1} \{(l-1)A(n,l)+(l-2)A(n,l-1 )\}(-v)^{n-l+1}T(v)^{l-1} -\sum\limits_{l=1}^n(l-1)A(n,l)(-v)^{n-l}T(v)^{l-1}
\end{align*}\end{small}

Combining them together, we have
\begin{align*}
&T(v)-vT'(v)+\cdots+\frac{(-v)^n}{n!}T^{(n)}(v)+\frac{(-v)^{n+1}}{(n+1)!}T^{(n+1)}(v)\\
=&(T(v)-vT'(v))\sum\limits_{l=1}^{n+1}  \frac{l}{n+1}(A(n,l)+A(n,l-1 )) (-v)^{n-l+1}T(v)^{l-1}\\
=&\sum\limits_{l=1}^{n+1}A( n+1,l)(-v)^{n+1-l}T(v)^{l-1}.
\end{align*}
\end{proof}

\subsection{The number $F_{n_1,n_2,\bI}(d_1,d_2)$}\label{F}{\ }\\
\indent

Let  \be j(x)=\frac{\sinh(x/2)}{x/2}\ee be the
$j$-function (c.f. \cite[p. 167]{BGV}). Denote $Q(x):=j^{-1}(x).$ It is not hard to see that
\be Q(x)=e^{-\frac{x}{2}}\frac{x}{1-e^{-x}}=e^{-\frac{x}{2}}T(x).\ee

Denote
\be
\begin{split}
&F_{n_1,n_2, \bI}(d_1,d_2)\\
=&\sum_{\substack{0\leq r\leq n_2\\ \forall 1\leq j\leq r\\  l_j\geq 1,\sum l_j\leq n_1,0\leq m_j\leq l_j\\ 1\leq s_1<s_2<\cdots\ s_r \leq n_2}}
(-1)^{\sum_{j=1}^rm_j}\binom{\vec l}{ \vec m}\binom{\frac{d_1+n_1-1+\sigma_1}{2}-\vec s\cdot \vec m}{n_1}\binom{\frac{d_2+n_2-1 }{2}+\sum_{j=1}^rl_j-r}{n_2+\sum_{j=1}^rl_j},
\end{split}
\ee
where we denote $\vec l=(l_1, \cdots, l_r), \vec m=(m_1, \cdots, m_r), \vec s=(i_{s_1},\cdots, i_{s_r})$ and
\be \binom{\vec l}{\vec m}\coloneqq\binom{l_1}{m_1}\cdots\binom{l_r}{m_r}, \  \vec s\cdot \vec m\coloneqq\sum_{j=1}^rm_ji_{s_j}. \ee


\begin{proposition}\label{mainprop} Let $V, u, v, d_1, d_2$ be as in Section 2. One has
$$F_{n_1, n_2, \bI}(d_1,d_2)=\langle Q(u)^{n_1+1}Q(v)\prod_{j=1}^{n_2}Q(v-i_ju)e^{\frac{d_1u+d_2v }{2}},[V]\rangle.$$
\end{proposition}

The rest of this subsection is devoted to the proof of this proposition. We will proceed by two steps.

In the expansion of $\widehat{A}(M)\cdot e^{\frac{d_1u+d_2v }{2}}$ with respect to $u$, observe that only $n_{1}+1$ terms
$$\{u^{n_1}v^{n_{2}},u^{n_1-1}v^{n_{2}+1},\cdots, v^{n_{2}+2}\}$$
survive, as $u^{m}=0,\forall\ m>n_1$.

\vskip.2cm
Let  $\gamma_1,\gamma_2$ be small circles around $u=0$ resp. $ v=0$, we have
\be
\begin{split}
&\langle Q(u)^{n_1+1}Q(v)\prod_{j=1}^{n_2}Q(v-i_ju) e^{\frac{d_1u+d_2v }{2}},[V]\rangle\\
=&\sum_{k=0}^{n_{1}} \beta_{ k}\cdot(\frac{1}{2\pi i})^2\oint_{\gamma_1}\oint_{\gamma_2}\frac{
Q(u)^{n_1+1}Q(v)\prod_{j=1}^{n_2}Q(v-i_ju) e^{\frac{d_1u+d_2v }{2}}
}{u^{n_1-k+1}v^{n_{2}+k+1}}dudv.
\end{split}
\ee

\vskip.3cm
Recall that $d_1u+d_2v=(2k_{1}+n_1+1-\sigma_{1})u+(2k_{2}+n_{2}+1)v$ and $Q(x)=e^{-\frac{x}{2}}T(x).$ So

\begin{align*}
&\langle Q(u)^{n_1+1}\prod_{j=1}^{n_2}Q(v-i_ju)Q(v)e^{\frac{d_1u+d_2v }{2}},[V]\rangle\\
=&\langle Q(u)^{n_1+1}Q(v)\prod\limits_{j=1}^{n_2}Q(v-i_ju)e^{\frac{d_1u+d_2v }{2}},[V]\rangle\\
=&\langle
T(u)^{n_{1}+1}T(v)\prod\limits_{j=1}^{n_2}T(v-i_ju)e^{\frac{d_1u+d_2v }{2}-\frac{v}{2}-\frac{(n_{1}+1)u}{2}-\frac{\sum_{j=1}^{n_2}(v-i_ju)}{2}}
, [V] \rangle\\
=& \langle
T(u)^{n_{1}+1}T(v)\prod\limits_{j=1}^{n_2}T(v-i_ju)e^{k_{1}u+k_{2}v}
, [V] \rangle \\
=& \langle
T(u)^{n_{1}+1}e^{k_{1}u}T(v)\prod\limits_{j=1}^{n_2}\{T(v)-i_juT'(v)+\cdots+\frac{(-i_ju)^{n_{1}}}{n_{1}!}T^{(n_{1})}(v)\}e^{k_{2}v}
, [V] \rangle,
\end{align*}
where the last equation used the Taylor expansion :
$$T(v-iu)\equiv T(v)+\frac{-i}{1!}T'(v)u+\cdots+\frac{(-i)^{n_1}}{n_1!}T^{(n_1)}(v)u^{n_1} \ \mod \ u^{n_1+1} .$$

\subsubsection{The first step} In this step, to make the notations simpler, denote
$$b_{m}:=\frac{1}{2\pi i}\oint_{\gamma}\frac{T(u)^{n_1+1}e^{k_1u}}{u^{m+1}}du$$
and thus
$$T(u)^{n_{1}+1}e^{k_{1}u}=1+b_1u+\cdots+b_m u^m+\cdots,$$
where $\gamma$ is a small circle around $u=0$, \ $0\leq m \leq n_1$.

\begin{lemma}\label{integral}
$$b_{n_1}=\frac{1}{2\pi i}\oint_{\gamma}\frac{T(u)^{n_1+1}e^{ku}}{u^{n_1+1}}du=\binom {n_1+k}{n_1},$$
where $\binom {n_1+k}{n_1}$ is the generalised binomial coefficient defined as
$$\binom {n_1+k} {n_1}=\frac{(n_1+k)\cdots(k+1)}{n_1!}, \ \ \forall\ n_1\in\N, \ k\in\Z.$$
\end{lemma}
\begin{proof}
\begin{align*}
&\frac{1}{2\pi i}\oint_{\gamma}\frac{T(u)^{n_1+1}e^{ku}}{u^{n_1+1}}du\\
=&\frac{1}{2\pi i}\oint_{\gamma}(\frac{u}{1-e^{-u}})^{n_1+1}e^{ku}\frac{1}{u^{n_1+1}}du\\
=&\frac{1}{2\pi i}\oint_{\gamma}\frac{e^{(n_1+k+1)u}}{(e^{u}-1)^{n_1+1}}du\\
=&\frac{1}{2\pi i}\oint_{\gamma}\frac{(1+t)^{n_1+k}}{t^{n_1+1}}dt\ (let\ e^{u}=t+1)\\
=&\binom {n_1+k}{n_1}.
\end{align*}
\end{proof}

Furthermore,
\begin{align*}
&\langle Q(u)^{n_1+1}\prod_{j=1}^{n_2}Q(v-i_ju)Q(v)e^{\frac{d_1u+d_2v }{2}},[V]\rangle\\
=& \langle
(\sum\limits_{m=0}^{n_{1}}b_{m}u^{m})T(v)\prod\limits_{k=1}^{n_2}\{\sum\limits_{p_k=0}^{n_1}\frac{(-i_ku)^{p_k}}{p_k!}T^{(p_k)}(v)\}e^{k_{2}v}
, [V] \rangle\\
=&\langle
(\sum\limits_{m=0}^{n_{1}}b_{m}u^{m})T(v)\{\sum\limits_{p=0}^{n_1}u^p\sum\limits_{\substack{p_1+\cdots+p_{n_2}=p\\ \forall 1\leq k\leq n_2, p_k\geq 0}}(-1)^p \cdot i_1^{p_1}\cdots i_{n_{2}}^{p_{n_2}}\frac{T^{(p_1)}(v)}{p_1!}\cdots\frac{T^{(p_{n_2})}(v)}{p_{n_2}!}\}e^{k_{2}v}
, [V] \rangle\\
=&\langle
\sum\limits_{m=0}^{n_{1}}u^m\{\sum\limits_{p=0}^{m}b_{m-p}\sum\limits_{\substack{p_1+\cdots+p_{n_2}=p\\ \forall 1\leq k\leq n_2, p_k\geq 0}}(-1)^p\cdot i_1^{p_1}\cdots i_{n_{2}}^{p_{n_2}}\frac{T^{(p_1)}(v)}{p_1!}\cdots\frac{T^{(p_{n_2})}(v)}{p_{n_2}!}\}T(v)e^{k_{2}v}
, [V] \rangle\\
=&\sum\limits_{m=0}^{n_{1}}\beta_{n_{1}-m}\cdot\{\sum\limits_{p=0}^{m}b_{m-p}\sum\limits_{\substack{p_1+\cdots+p_{n_2}=p\\ \forall 1\leq k\leq n_2, p_k\geq 0}}(-1)^p\cdot i_1^{p_1}\cdots i_{n_{2}}^{p_{n_2}}
\frac{1}{2\pi i}\oint_{\gamma}\frac{
\frac{T^{(p_1)}(v)}{p_1!}\cdots\frac{T^{(p_{n_2})}(v)}{p_{n_2}!}
}{v^{n_2+n_1-m+1}}T(v)e^{k_{2}v}dv\}\\
=&\sum\limits_{m=0}^{n_{1}}\sum\limits_{p=0}^{m}b_{m-p}\sum\limits_{\substack{p_1+\cdots+p_{n_2}=p \\ q_1+\cdots+q_{n_2}=n_1-m\\ \forall 1\leq k\leq n_2, p_k, q_k\geq 0}}i_1^{q_1}\cdots i_{n_{2}}^{q_{n_2}}\cdot i_1^{p_1}\cdots i_{n_{2}}^{p_{n_2}}\frac{(-1)^p}{2\pi i}\oint_{\gamma}\frac{
\{ \frac{T^{(p_1)}(v)}{p_1!}\cdots\frac{T^{(p_{n_2})}(v)}{p_{n_2}!}\}T(v)e^{k_{2}v}
}{v^{n_2+n_1+1-m}}dv\\
=&\sum\limits_{m=0}^{n_{1}}\sum\limits_{p=0}^{m}b_{m-p}\sum\limits_{\substack{p_1+\cdots+p_{n_2}=p\\ q_1+\cdots+q_{n_2}=n_1-m+p\\ \forall 1\leq k\leq n_2, q_k\geq p_k\geq 0}}i_1^{q_1}\cdots i_{n_{2}}^{q_{n_2}}
\frac{(-1)^p}{2\pi i}\oint_{\gamma}\frac{
\{ \frac{T^{(p_1)}(v)}{p_1!}\cdots\frac{T^{(p_{n_2})}(v)}{p_{n_2}!}\}T(v)e^{k_{2}v}
}{v^{n_2+n_1+1-m}}
dv
\end{align*}

\begin{align*}
=&\sum\limits_{t=0}^{n_{1}}b_{t}\sum\limits_{m=t}^{n_1}\sum\limits_{\substack{ q_1+\cdots+q_{n_2}=n_1-t\\ \forall 1\leq k\leq n_2,q_k\geq 0 }}
\sum\limits_{\substack{ p_1+\cdots+p_{n_2}=m-t\\  p_k\leq q_k, \forall 1\leq k\leq n_2}}i_1^{q_1}\cdots i_{n_{2}}^{q_{n_2}}\\
&\cdot
\frac{1}{2\pi i}\oint_{\gamma}\frac{
\{\frac{T^{(p_1)}(v)}{p_1!}\cdots\frac{T^{(p_{n_2})}(v)}{p_{n_2}!}\}(-v)^{m-t}T(v)e^{k_{2}v}
}{v^{n_2+n_1+1-t}}dv \ (\mathrm{Let}\ \ t=m-p)\\
=&\sum\limits_{t=0}^{n_{1}}\frac{ b_t}{2\pi i}\sum\limits_{\substack{\sum_{k=1}^{n_2}q_k=n_1-t\\ \forall 1\leq k\leq n_2,q_k\geq 0 }}
i_1^{q_1}\cdots i_{n_{2}}^{q_{n_2}}\sum\limits_{m=t}^{n_1}\sum\limits_{\substack{\sum p_k=m-t\\ p_k\leq q_k}}\oint_{\gamma}\frac{
\{
\frac{(-v)^{p_1}T^{(p_1)}(v)}{p_1!}\cdots\frac{(-v)^{p_{n_2}}T^{(p_{n_2})}(v)}{p_{n_2}!}\}T(v)e^{k_{2}v}
}{v^{n_2+n_1+1-t}}dv.
\end{align*}

Note that for fixed integers $q_1,\cdots, q_{n_2}$ with $\sum_{k=1}^{n_2}q_k=n_1-t$, we have
\begin{align*}
&\sum\limits_{m=t}^{n_1}\sum\limits_{\substack{\sum p_k=m-t\\  p_k\leq q_k}}\frac{(-v)^{p_1}T^{(p_1)}(v)}{p_1!}\cdots\frac{(-v)^{p_{n_2}}T^{(p_{n_2})}(v)}{p_{n_2}!}\\
=&\sum\limits_{N=0}^{n_1-t}\sum\limits_{\substack{\sum p_k=N\\  p_k\leq q_k}}\frac{(-v)^{p_1}T^{(p_1)}(v)}{p_1!}\cdots\frac{(-v)^{p_{n_2}}T^{(p_{n_2})}(v)}{p_{n_2}!}\\
=&\prod\limits_{\substack{1\leq k\leq n_2}}(T(v)-vT'(v)+\cdots+\frac{(-v)^{q_k}T^{(q_k)}(v)}{q_k!})\\
=&\prod\limits_{\substack{1\leq k\leq n_2}}\sum\limits_{m_k=0}^{q_k}\frac{(-v)^{m_k}}{m_k!}T^{(m_k)}(v).
\end{align*}

Thus
\begin{equation}\label{observe}
\begin{aligned}
&\langle Q(u)^{n_1+1}\prod_{j=1}^{n_2}Q(v-i_ju)Q(v)e^{\frac{d_1u+d_2v }{2}},[V]\rangle\\
=&\sum\limits_{t=0}^{n_{1}}b_t\sum\limits_{\substack{\sum q_k=n_1-t\\ \forall 1\leq k\leq n_2,q_k\geq 0}}i_1^{q_1}\cdots i_{n_{2}}^{q_{n_2}}\frac{1}{2\pi i}\oint_{\gamma}\frac{
\prod\limits_{k=1}^{n_{2}}(\sum\limits_{m_k=0}^{q_k}\frac{(-v)^{m_k}}{m_k!}T^{(m_k)}(v))
}{v^{n_2+1+(n_1-t)}}T(v)e^{k_{2}v}dv\\
=&\sum\limits_{t=0}^{n_{1} }b_t\sum\limits_{r=0}^{n_2}\sum\limits_{ 1\leq s_1<s_2<\cdots\ s_r \leq n_2}i_{s_1}^{p_1}\cdots i_{s_r}^{p_{r}}\frac{1}{2\pi i}\oint_{\gamma}\frac{
\prod\limits_{k=1}^{r}(\sum\limits_{m_k=0}^{p_k}\frac{(-v)^{m_k}}{m_k!}T^{(m_k)}(v))
}{v^{n_2+1+n_1-t}}T(v)^{n_2-r+1}e^{k_{2}v}dv,
\end{aligned}
\end{equation}
where $\{p_1,\cdots,p_r\}$ be a positive partition of $n_1-t$.

\begin{lemma}\label{10}
Let $\{p_1,\cdots,p_r\}$ be a positive partition of $n_1-t$. Then
$$
\frac{1}{2\pi i}\oint_{\gamma}\frac{\prod_{k=1}^{r}(\sum\limits_{m_k=0}^{p_k}\frac{(-v)^{m_k}}{m_k!}T^{(m_k)}(v))}
{v^{n_2+1+n_1-t}}T(v)^{n_2+1-r}e^{k_{2}v}dv$$$$
=(-1)^{n_1-t}\sum_{\substack{\forall 1\leq k\leq r, 1\leq l_k\leq p_k\\l=l_1+\cdots+l_r}}A(p_1,l_1)\cdots A(p_r,l_r)(-1)^{l}\binom{n_2+l+k_2-r}{n_2+l}.
$$
\end{lemma}

\begin{proof}
By Lemma \ref{deriviate of g},
$$
\sum_{m_k=0}^{p_k}\frac{(-v)^{m_k}}{m_k!}T^{(m_k)}(v)=(T(v)-vT'(v))(\sum_{l_k=1}^{p_k}A(p_k,l_k)(-v)^{p_k-l_k}T(v)^{l_k-1}).
$$
Thus\begin{small}
\begin{align*}
&\frac{1}{2\pi i}\oint_{\gamma}\frac{\prod\limits_{k=1}^{r}(\sum\limits_{m_k=0}^{p_k}\frac{(-v)^{m_k}}{m_k!}T^{(m_k)}(v))}
{v^{n_2+1+p}}T(v)^{n_2+1-r}e^{k_{2}v}dv\\
=&\frac{1}{2\pi i}\oint_{\gamma}\frac{\prod\limits_{k=1}^{r}(\sum\limits_{l_k=1}^{p_k}A(p_k,l_k)(-v)^{p_k-l_k}T(v)^{l_k-1})}
{v^{n_2+1+p}}(T(v)-vT'(v))^rT(v)^{n_2+1-r}e^{k_{2}v}dv\\
=&\sum\limits_{\substack{\forall 1\leq k\leq r, 1\leq l_k\leq p_k\\l=l_1+\cdots+l_r}}\frac{1}{2\pi i}\oint_{\gamma}\frac{A(p_1,l_1)\cdots A(p_r,l_r)(-v)^{p-l}T(v)^{l-r}}
{v^{n_2+1+p}}(T(v)-vT'(v))^r\\
&\cdot T(v)^{n_2+1-r}e^{k_{2}v}dv\\
=&\sum\limits_{\substack{\forall 1\leq k\leq r,1\leq l_k\leq p_k\\l=l_1+\cdots+l_r}}A(p_1,l_1)\cdots A(p_r,l_r)\frac{(-1)^{p-l}}{2\pi i}\oint_{\gamma}\frac{(T(v)-vT'(v))^rT(v)^{n_2+1+l }e^{k_{2}v}}
{T(v)^{2r}v^{n_2+1+l}}dv\\
=&\sum\limits_{\substack{\forall 1\leq k\leq r, 1\leq l_k\leq p_k\\l=l_1+\cdots+l_r}}A(p_1,l_1)\cdots A(p_r,l_r)\frac{(-1)^{p-l}}{2\pi i}\oint_{\gamma}\frac{(\frac{v^2e^v}{(e^v-1)^2})^rT(v)^{n_2+1+l }e^{k_{2}v}}
{(\frac{ve^v}{e^v-1})^{ 2r}v^{n_2+1+l}}dv\\
=&\sum\limits_{\substack{\forall 1\leq k\leq r, 1\leq l_k\leq p_k\\l=l_1+\cdots+l_r}}A(p_1,l_1)\cdots A(p_r,l_r)(-1)^{p-l}\binom{n_2+l+k_2-r}{n_2+l}. \ ( \mathrm{By\ Lemma}\ \ref{integral})
\end{align*}
\end{small}
\end{proof}

To sum up
\begin{align*}
&\langle Q(u)^{n_1+1}\prod_{j=1}^{n_2}Q(v-i_ju)Q(v)e^{\frac{d_1u+d_2v }{2}},[V]\rangle\\
=&\sum\limits_{t=0}^{n_{1} }b_t \sum\limits_{\substack{0\leq r\leq n_2\\p=p_1+\cdots+p_r=n_1-t\\ \forall 1\leq k\leq r, p_k\geq 1\\ 1\leq s_1<s_2<\cdots\ s_r \leq n_2}} i_{s_1}^{p_1}\cdots i_{s_r}^{p_{r}}(-1)^p
\sum\limits_{\substack{\forall 1\leq k\leq r\\ 1\leq l_k\leq p_k\\l=l_1+\cdots+l_r}}A(p_1,l_1)\cdots A(p_r,l_r)(-1)^{l}\binom{n_2+k_2+l-r}{n_2+l}\\
=&\sum\limits_{\substack{0\leq r\leq n_2\\p=p_1+\cdots+p_r\leq n_1\\  \forall 1\leq k\leq r,p_k\geq 1\\ 1\leq s_1<s_2<\cdots\ s_r \leq n_2}}b_{n_1-p}i_{s_1}^{p_1}\cdots i_{s_r}^{p_{r}}(-1)^p
 \sum\limits_{\substack{\forall 1\leq k\leq r\\ 1\leq l_k\leq p_k\\l=l_1+\cdots+l_r}}A(p_1,l_1)\cdots A(p_r,l_r)(-1)^{l}\binom{n_2+k_2+l-r}{n_2+l}\\
 =&\sum\limits_{\substack{0\leq r\leq n_2\\\forall 1\leq k\leq r,1\leq l_k\\l=l_1+\cdots+l_r\leq n_1\\1\leq s_1<s_2<\cdots\ s_r \leq n_2}}(-1)^{l}\binom{n_2+k_2+l-r}{n_2+l}
 \sum\limits_{\substack{p_1+\cdots+p_r\leq n_1\\  \forall 1\leq k\leq r,p_k\geq l_k}}  \!\!\! b_{n_1-p}(-i_{s_1})^{p_1}A(p_1,l_1)\cdots(-i_{s_r})^{p_{r}}A(p_r,l_r)
\end{align*}

\subsubsection{The second step}

\begin{lemma}\label{5}
For any fixed $\{l_1,\cdots,l_r,s_1,\cdots,s_r\}$, let $l=l_1+\cdots+l_r$, we have
\begin{align*}
&\sum_{\substack{p_1+\cdots+p_r\leq n_1\\  p_k\geq l_k,\forall 1\leq k\leq r}}  b_{n_1-p}(-i_{s_1})^{p_1}A(p_1,l_1)\cdots(-i_{s_r})^{p_{r}}A(p_r,l_r)\\
=&\sum_{\substack{1\leq j\leq r\\ 0\leq m_j\leq l_r\\ m=m_1+\cdots+m_r}}(-1)^{l-m}\binom{\vec l}{\vec{m}}\binom{n_1+k_1-\vec s\cdot \vec m}{n_1} ,
\end{align*}
where  $\vec l=(l_1, \cdots, l_r), \vec m=(m_1, \cdots, m_r), \vec s=(i_{s_1},\cdots, i_{s_r})$
and
$ \binom{\vec l}{\vec m}\coloneqq\binom{l_1}{m_1}\cdots\binom{l_r}{m_r},$ \\ $\vec s\cdot \vec m \coloneqq\sum\limits_{j=1}^rm_j\cdot i_{s_j}. $
\end{lemma}

\begin{proof}
$$b_{n_1-p}=\frac{1}{2\pi i}\oint_{\gamma}\frac{T(u)^{n_1+1}e^{k_1u}}{u^{n_1-p+1}}du=\frac{1}{2\pi i}\oint_{\gamma}\frac{u^{p_1}\cdots u^{p_r}T(u)^{n_1+1}e^{k_1u}}{u^{n_1+1}}du.$$
Thus
\begin{align*}
&\sum\limits_{\substack{p_1+\cdots+p_r\leq n_1\\  p_k\geq l_k,\forall 1\leq k\leq r}}  b_{n_1-p}\cdot(-i_{s_1})^{p_1}A(p_1,l_1)\cdots(-i_{s_r})^{p_{r}}A(p_r,l_r)\\
=&\sum\limits_{\substack{p_1+\cdots+p_r\leq n_1\\  p_k\geq l_k,\forall 1\leq k\leq r}}  \frac{1}{2\pi i}\oint_{\gamma}\frac{
  (-i_{s_1}u)^{p_1}A(p_1,l_1)\cdots (-i_{s_r}u)^{p_{r}}A(p_r,l_r) }{u^{n_1+1}}T(u)^{n_1+1}e^{k_1u}du.
\end{align*}

Since $(-i_{s_k}u)^{p_k}A(l_k,p_k)=0, \ p_k>n_1$
\begin{align*}
=&\sum\limits_{\substack{  p_k\geq l_k,\forall 1\leq k\leq r\\ }}  \frac{1}{2\pi i}\oint_{\gamma}\frac{
  (-i_{s_1}u)^{p_1}A(p_1,l_1)\cdots (-i_{s_r}u)^{p_{r}}A(p_r,l_r) }{u^{n_1+1}}T(u)^{n_1+1}e^{k_1u}du\\
  =&\frac{1}{2\pi i}\oint_{\gamma}\frac{
  \prod\limits_{j=1}^{r}\{\sum\limits_{m_j=l_j}^{\infty}(-i_{s_j}u)^{m_j} A(m_j,l_j)\}}{u^{n_1+1}}T(u)^{n_1+1}e^{k_1u}du.
\end{align*}

By Lemma \ref{3},

\begin{align*}
&\sum\limits_{m=l}^{\infty} u^mA(m,l)=(e^{ u}-1)^l\\
=&\frac{1}{2\pi i}\oint_{\gamma}\frac{ \prod\limits_{j=1}^{r}(e^{-i_{s_j}u}-1)^{l_j}}{u^{n_1+1}}T(u)^{n_1+1}e^{k_1u}du\\
=&\frac{1}{2\pi i}\oint_{\gamma}\frac{
  \prod\limits_{j=1}^{r}\sum_{m_j=0}^{l_j}(-1)^{l_j-m_j}\binom{l_j}{ m_j}e^{-m_ji_{s_j}u}}{u^{n_1+1}}T(u)^{n_1+1}e^{k_1u}du\\
=&\sum\limits_{m_1=0}^{l_1}\cdots\sum\limits_{m_r=0}^{l_r}(-1)^{\sum\limits_{j=1}^rl_j-m_j}
\binom{l_1}{ m_1}\cdots \binom{l_r}{ m_r}\frac{1}{2\pi i}\oint_{\gamma}\frac{e^{-(\sum\limits_{j=1}^rm_ji_{s_j})u}}{u^{n_1+1}}T(u)^{n_1+1}e^{k_1u}du\\
=&\sum_{\substack{1\leq j\leq r\\ 0\leq m_j\leq l_r\\m=\sum_{j=1}^rm_j}}(-1)^{l-m}\binom{\vec l}{\vec m}\binom{n_1+k_1-\vec s\cdot \vec m}{n_1}. \ (\mathrm{By\ Lemma}\ \ref{integral})
\end{align*}

\end{proof}

Combining Lemma \ref{10} and Lemma \ref{5}, we deduce Proposition \ref{mainprop}.


\subsection{$\hat{A}$-genus and $\alpha$-invariant} Recall that in Section 2, we have shown that
$$ c(H_{n_1,n_2}^{\bI}(d_1,d_2))=i^*(c(V)(1+d_1u+d_2v)^{-1}),$$
$$ c(V)=(1+u)^{n_{1}+1}(1+v)\prod_{j=1}^{n_2}(1+v-i_{j}u).$$
Therefore
$$ p(H_{n_1,n_2}^{\bI}(d_1,d_2))=i^*\{(1+u^2)^{n_{1}+1}(1+v^2)\prod_{j=1}^{n_2}(1+(v-i_{j}u)^2)(1+(d_1u+d_2v)^2)^{-1}\}.$$

The characteristic power series of the $\widehat{A}$-genus is just $\frac{\sqrt{z}/2}{sinh(\sqrt{z}/2)}=Q(x)$ for $z=x^2$.

For $n_1+n_2\equiv\ 1\ \mod\ 2$, by Poincar\'e duality and Proposition \ref{mainprop}, one has
\be
\begin{split}
&\widehat{A}(H_{n_1,n_2}^{\bI}(d_1,d_2))\\
= &\langle Q(u)^{n_1+1}Q(v)\prod_{j=1}^{n_2}Q(v-i_ju)\cdot Q(d_1u+d_2v)^{-1}, [V]\cap(d_1u+d_2v)\rangle \\
=&\langle Q(u)^{n_1+1}Q(v)\prod_{j=1}^{n_2}Q(v-i_ju)\cdot \frac{e^{d_1u+d_2v}-1}{(d_1u+d_2v) e^{d_1u+d_2v}}e^{\frac{d_1u+d_2v}{2}}(d_1u+d_2v),[V]\rangle\\
=&\langle Q(u)^{n_1+1}Q(v)\prod_{j=1}^{n_2}Q(v-i_ju)\cdot(e^{\frac{d_1u+d_2v}{2}}-e^{\frac{-d_1u-d_2v}{2}}),[V]\rangle\\
=&F_{n_1,n_2,\bI}(d_1,d_2)-F_{n_1,n_2,\bI}(-d_1,-d_2).\\
\end{split}
\ee
So we have
\begin{theorem}\label{Ahat}
For  $n_1+n_2\equiv\ 1\ \mod\ 2$ and $d_1, d_2 \in \Z $ .
$$\widehat{A}(H_{n_1,n_2}^{\bI}(d_1,d_2))=F_{n_1,n_2,\bI}(d_1,d_2)-F_{n_1,n_2,\bI}(-d_1 ,-d_2 ).$$
\end{theorem}

$\, $

The $\alpha$-invariant (\S6, Chapter V in \cite{Bo66}) is a ring homomorphism
$$\alpha: \Omega_*^{Spin}\to KO_*(pt). $$

As we all know, $KO_{*}$ is 8 periodic,
\begin{equation*}
KO_{n}(pt)=\begin{cases}
0, & \mathrm{for} \ n\equiv 3,5,6,7\ \mod\ 8,\\
\Z, &  \mathrm{for}\  n\equiv 0, 4\ \mod\ 8, \\
\Z_2, &  \mathrm{for}\ n\equiv 1, 2\ \mod\ 8,
\end{cases}
\end{equation*}
and
\begin{equation*}
\alpha(M)=\begin{cases}
\widehat{A}(M), & \mathrm{for} \ n\equiv 0\ \mod\ 8,\\
\frac{1}{2}\widehat{A}(M) &  \mathrm{for}\  n\equiv 4\ \mod\ 8.

\end{cases}
\end{equation*}
The $\alpha$-invariant of an $8k+2$ dimensional manifold is the  mod $ 2$ index of Atiyah-Singer Dirac operator, which is difficult to compute. To perform the computation, we use the following formula.

\begin{theorem}[\protect Zhang \cite{Zh93,Zh94,Zh96R}, c.f. \cite{Zh96E}] \label{Zhang-Rokhlin}
Let $M$ be a compact connected spin$^{c}$-manifold of dimension $8k+4$, $\xi$ is a complex line bundle on $M$ with $c_1(\xi)\equiv \omega_2(TM)$. $B$ is the spin submanifold of $M$ \poincare \ dual to $c_1(\xi)\in H^{2}(M;\Z)$, then $B$ carries an induced spin structure and

$$\alpha (B)\equiv\langle \widehat{A}(M)exp^{\frac{c_1(\xi)}{2}},[M]\rangle\ \mod \ 2. $$

\end{theorem}

Thus when dim$_{\R}H_{n_1,n_2}^{\bI}(d_1,d_2)\equiv 2\ \mod\ 8$, applying Zhang's theorem, we have
\be
\begin{split}
&\alpha(H_{n_1,n_2}^{\bI}(d_1,d_2))\\
\equiv& \langle \widehat{A}(V)e^{\frac{d_1u+d_2v}{2}}, [V] \rangle \\
\equiv&
\langle Q(u)^{n_1+1}\prod_{j=1}^{n_2}Q(v-i_ju)Q(v)e^{\frac{d_1u+d_2v }{2}},[V]\rangle \ \mod\ 2.
\end{split} \ee

By Proposition \ref{mainprop}, we have
\begin{theorem}\label{13}
If $H_{n_1,n_2}^{\bI}(d_1,d_2)$ is spin and  $n_1+n_2\equiv2\ \mod\ 4$, then
$$\alpha(H_{n_1,n_2}^{\bI}(d_1,d_2))\equiv F_{n_1,n_2,\bI}(d_1,d_2)\ \mod\ 2.$$
\end{theorem}

\subsection{Some examples}\label{examples}{\ }\\
\indent
In this subsection, we assume $n_1+n_2 \equiv 1\ \mod \ 2$ when discussing $\widehat{A}$-genus and $n_1+n_2 \equiv 2\ \mod \ 4$ when discussing $\alpha$-invariant.
\vskip.2cm
{\bf Observation:}
From formula (\ref{observe}), we can observe that if $r$ is greater than the number of non-zero $i_j$, the summand vanishes.

In particular, if $\bI =\bf{0}$ or $(j,0,\cdots,0)$, $F_{n_1, n_2, \bI}(d_1,d_2)$ can be simplified to a great extent. More precisely, if $\bI =\bf{0}$, the summand is non vanishing only when $r=0$; if $(j,0,\cdots,0)$,  the summand is non vanishing only when $r=0,1$.
\vskip.3cm

By above observation, we have
\be \label{Ffac} F_{n_1,n_2,\mathbf{0}}(d_1,d_2)=\binom{n_1+k_1}{n_1}\binom{n_2+k_2}{n_2}.\ee

\begin{example}\label{3.1}
When $\bI=\mathbf{0}$, $V=\CC P^{n_1}\times\CC P^{n_2}$, $H_{n_1,n_2}^{\mathbf{0}}( 1,1)$ is the usual Milnor hypersurface. We have
\begin{equation}
\begin{aligned}
\widehat{A}(H_{n_1,n_2}^{\mathbf{0}}( 1,1))
&=F_{n_1,n_2,\mathbf{0}}( 1,1)-F_{n_1,n_2,\mathbf{0}}( -1,-1)\\
&=\binom{ \frac{n_1}{2}}{n_1}\binom{ \frac{n_2}{2}}{n_2}-\binom{ \frac{n_1}{2}-1}{n_1}\binom{ \frac{n_2}{2}-1}{n_2}\\
&=\binom{\frac{n_1}{2}}{n_1}\binom{\frac{n_2}{2}}{n_2}-(-1)\binom{\frac{n_1}{2}}{n_1}(-1)\binom{\frac{n_2}{2}}{n_2}\\
&=0
\end{aligned}
\end{equation}
and
\begin{equation}
\alpha(H_{n_1,n_2}^{\mathbf{0}}( 1,1))\equiv \binom{ \frac{n_1}{2}}{n_1}\binom{ \frac{n_2}{2}}{n_2}\equiv 0\ \mod\ 2,
\end{equation}
which coincides with the result in \cite[p. 40]{HBJ94} that the $\widehat{A}$-genus of Milnor hypersurface always vanishes.
\end{example}

\begin{example}\label{3.2}
When $\bI=\mathbf{0}$, $V=\CC P^{n_1}\times\CC P^{n_2}$, we have
\begin{align*}
\widehat{A}(H_{n_1,n_2}^{\mathbf{0}}(d_1,d_2))
&=F_{n_1,n_2,\bI=\mathbf{0}}(d_1,d_2)-F_{n_1,n_2,\bI=\mathbf{0}}(-d_1,-d_2)\\
&=\binom{ \frac{d_1+n_1-1}{2}}{n_1}\binom{ \frac{d_2+n_2-1}{2}}{n_2}-\binom{ \frac{-d_1+n_1-1}{2}}{n_1}\binom{ \frac{-d_2+n_2-1}{2}}{n_2}\\
&=\binom{n_1+k_1}{n_1}\binom{n_2+k_2}{n_2}-\binom{-k_1-1}{n_1}\binom{-k_2-1}{n_2}\\
&=(1-(-1)^{n_1+n_2})\binom{n_1+k_1}{n_1}\binom{n_2+k_2}{n_2}\\
&=2\binom{n_1+k_1}{n_1}\binom{n_2+k_2}{n_2}
\end{align*}

 \vskip.3cm

When $k_1=-\frac{n_1+1-d_1}{2},k_2=-\frac{n_2+1-d_2}{2}$ are integers,  $H_{n_1,n_2}^{\mathbf{0}}(d_1,d_2)$ is spin and
\begin{enumerate}
  \item $\widehat{A}(H_{n_1,n_2}^{\mathbf{0}}(d_1,d_2))=0\iff  -n_i < k_i< 0,\ \text{i.e.} \ -n_i\leq d_i\leq n_i, \ i=1,2$;
  \item $ \alpha(H_{n_1,n_2}^{\mathbf{0}}(d_1,d_2))\equiv \binom{n_1+k_1}{n_1}\binom{n_2+k_2}{n_2}\ \mod\ 2.$
\end{enumerate}

\end{example}

\begin{example}\label{3.4}
When $\bI=(j,0,\cdots,0), \ j\in \Z$,\ By the previous observation,
\begin{small}
\begin{align*}
&\widehat{A}(H_{n_1,n_2}^{(j,0,\cdots,0)}(d_1,d_2))=2\binom{n_1+k_1 }{n_1}\binom{n_2+k_2 }{n_2}+\sum_{l=1}^{n_1}\sum_{m=0}^{l}(-1)^{m }\binom{l}{m}\\
&\cdot\left\{\binom{n_1+k_1-mj }{n_1}\binom{n_2+k_2+l-1}{n_2+l}-\binom{-k_1-1+(1-m)j}{n_1}\binom{-k_2-1+l-1}{n_2+l}\right\}.\\
&\ \\
&\alpha(H_{n_1,n_2}^{(j,0,\cdots,0)}(d_1,d_2))\\
&\equiv \binom{n_1+k_1}{n_1}\binom{n_2+k_2}{n_2}+\sum\limits_{l=1}^{n_1}\sum\limits_{m=0}^{l}(-1)^{m }\binom{l}{m}\binom{n_1+k_1-mj }{n_1}\binom{n_2+k_2+l-1}{n_2+l}.\\
\end{align*}
\end{small}
Assume $n_1=2, n_1+n_2\equiv 1\ \mod\ 2,\ d_1=1, d_2=n_2+1$, we have
$$\widehat{A}(H_{2,n_2}^{(j,0,\cdots,0)}(1,n_2+1)=\frac{j}{2}(\frac{j}{2}+1)$$
Thus $$\widehat{A}(H_{2,n_2}^{(j,0,\cdots,0)}(1,n_2+1)\neq0 \iff j\neq 0,-2.$$

Note that if  $ j\equiv 0 \ \mod\ 2$, then $H_{2,n_2}^{(j,0,\cdots,0)}(1,n_2+1)$ is spin.
  \vskip.2cm
Since $\widehat{A}(H_{2,n_2}^{\bf 0}(1,n_2+1)=0$,  this provides a good example to illustrate the difference between twisted Milnor hypersurface and non twisted one.

\end{example}

\vskip.5cm
\begin{example}\label{3.3}
Assume $d_1=d_2=1$ and $k_1=-\frac{n_1-\sigma_1}{2},k_2=-\frac{n_2}{2}$ are integers, then $H_{n_1,n_2}^{\bI}(1,1)$ is spin and
\be
\begin{aligned}
&\ \ F_{n_1,n_2, \bI}(1,1)\\
=&\sum_{  \substack{0\leq r\leq n_2\\ \forall 1\leq j\leq r\\  l_j\geq 1,l\leq n_1,0\leq m_j\leq l_j\\ 1\leq s_1<s_2<\cdots\ s_r \leq n_2}}
(-1)^{\sum_{j=1}^rm_j}\binom{\vec l}{ \vec m}\binom{n_1+\frac{\sigma_1-n_1}{2}-\vec s\cdot \vec m }{n_1}
\binom{n_2+\frac{-n_2}{2}+\sum_{j=1}^rl_j-r}{n_2+\sum_{j=1}^rl_j},
\end{aligned}
\ee
\noindent {\bf Claim:}
$$\binom{\frac{n_2}{2}+\sum_{j=1}^rl_j-r}{n_2+\sum_{j=1}^rl_j}=0.$$

In fact, let $m=\frac{n_2}{2}+\sum_{j=1}^rl_j-r, \ n=n_2+\sum_{j=1}^rl_j$, we have
\vskip.3cm
\begin{itemize}
  \item $m=\frac{n_2}{2}+\sum_{j=1}^rl_j-r$ positive integer.
  \vskip.3cm
  (since $,l_j\geq 1$, $\sum_{j=1}^{r}l_j- r\geq 0$ and $k_2=\frac{-n_2}{2}$ integer)
  \vskip.3cm
  \item $m<n$ both positive integers.
  \vskip.3cm
  ($\frac{n_2}{2}+\sum_{j=1}^rl_j-r <n_2+\sum_{j=1}^rl_j \ \iff -r <\frac{n_2}{2}$)
  \vskip.3cm
  \item $\binom{m}{n}=0$ for two positive integers $m<n$.
\end{itemize}
Therefore $ F_{n_1,n_2, \bI}(1,1)=0$. Similarly, $F_{n_1,n_2, \bI}(-1,-1)=0$. Thus
$$\widehat{A}(H_{n_1,n_2}^{\bI}(1,1))=0,\ \alpha(H_{n_1,n_2}^{\bI}(1,1))=0.$$

More generally, as long as $n_2$ is even, we have \ $\widehat{A}(H_{n_1,n_2}^{\bI}(d_1,1))=0$,\ $\alpha(H_{n_1,n_2}^{\bI}(d_1,1))=0$.
\end{example}

\begin{example}\label{3.5}
For $n_1=1$ and $ n_2\equiv 1\ \mod\ 4$, we have
\begin{small}
\begin{align*}
&\alpha(H_{n_1,n_2}^{\bI}(d_1,d_2))\\
\equiv& \sum\limits_{\substack{l =1\\1\leq s \leq n_2}}\sum\limits_{m =0}^{l }(-1)^{m  }\binom{l }{m }\binom{n_1+k_1-i_{s }m }{n_1}\binom{n_2+k_2+l -1}{n_2+l }+\binom{k_1+1}{1}\binom{n_2+k_2}{n_2}\ \mod\ 2.\\
\equiv& \sum_{1\leq s \leq n_2}\binom{n_2+k_2}{n_2+1}\sum_{m =0}^1(-1)^{m }(1+k_1-i_sm)+\binom{k_1+1}{1}\binom{n_2+k_2}{n_2}\ \mod\ 2.\\
\equiv& \binom{n_2+k_2}{n_2+1}\sum_{1\leq s \leq n_2}(1+k_1-(1+k_1-i_s))+\binom{k_1+1}{1}\binom{n_2+k_2}{n_2}\ \mod\ 2.\\
\equiv&\sigma_1\binom{n_2+k_2}{n_2+1}+\binom{k_1+1}{1}\binom{n_2+k_2}{n_2} \ \ \mod\ 2.
\end{align*}
\end{small}
\end{example}

\begin{example}\label{3.6}
For $n_1=2$ and  $n_2\equiv 0\ \mod\ 4$, we have
\begin{align*}
 &\alpha(H_{n_1,n_2}^{\bI}(d_1,d_2))\\
 \equiv& \binom{k_1+2}{2}\binom{n_2+k_2}{n_2}+\sum\limits_{\substack{1 \leq l \leq 2\\1\leq s \leq n_2}}\sum\limits_{m =0}^{l }(-1)^{m }\binom{l}{m}\binom{n_1+k_1-i_{s} m}{n_1}\binom{n_2+k_2+l-1}{n_2+l}+ \\
  &\!\sum\limits_{\substack{ 1 \leq l_1\leq 2\\ 1 \leq l_2\leq 2 \\1\leq s_1< s_2\leq n_2}}\!\!\!\sum\limits_{\substack{0 \leq m_1 \leq l_1\\ 0\leq m_2 \leq l_2}}(-1)^{m_1+m_2}\binom{l_1}{m_1}\binom{l_2}{m_2}\binom{n_1+k_1-i_{s_1}m_1-i_{s_2}m_2}{n_1}\binom{n_2+k_2+l_1+l_2-2}{n_2+l_1+l_2}\\
\equiv& \binom{k_1+2}{2}\binom{n_2+k_2}{n_2}+\left(-\frac{\sigma_1^2-2\sigma_2}{2}+\frac{(2k_1+3)\sigma_1}{2}\right)\binom{n_2+k_2}{n_2+1}\\
& \ \ +(\sigma_1^2-2\sigma_2)\binom{n_2+k_2+1}{n_2+2}+\sigma_2\binom{n_2+k_2}{n_2+2}\ \mod\ 2.
\end{align*}
\end{example}


\section{Applications}

In this section, we give some applications of $\widehat{A}$-genus and $\alpha$-invariant of twisted Milnor hypersurfaces.
\vskip.2cm
First, we investigate the existence of circle actions on twisted Milnor hypersurfaces.

\begin{theorem}[ Atiyah-Hirzebruch\cite{AH70}]
If the circle $S^1$ acts nontrivially on a connected spin manifold $M$, then $\widehat{A}(M)=0$.
\end{theorem}

Since we have given the sufficient and necessary conditions for non vanishing $\widehat{A}$-genus  about spin twisted Milnor hypersurfaces in Example \ref{3.2} and \ref{3.4}, we have
\begin{corollary}\indent
 \begin{enumerate}
 \item[$(1)$] Assume $\bI=\mathbf{0}$, $n_1+n_2$ is odd and $k_i\in \Z, k_i\leq-n_i$ or $k_i\geq0, i=1,2$. Then there does not exist non-trivial circle actions on $H_{n_1,n_2}^{\mathbf{0}}(d_1,d_2)$;
\item[$(2)$] Assume $n_2$ is odd and $j\neq 0,-2$ is even. Then there does not exist non-trivial circle action on the twisted spin Milnor hypersurface $H_{2,n_2}^{(j,0,\cdots,0)}(1,n_2+1)$.
\end{enumerate}
\end{corollary}

Example \ref{3.3} tells us  $\widehat{A}(H_{n_1,n_2}^{\bI}(d_1,1))=0, \forall \ \bI $. The smooth hypersurface $H_{n_1,n_2}^{\mathbf{0}}(d_1,1)$ can be described as the zero locus of equation
$$x_0^{d_1}y_0+x_1^{d_1}y_1+\cdots x_{n_1}^{d_1}y_{n_1}=0, \ (n_1\leq n_2 ) $$
where $[x_0: x_1: \cdots: x_{n_1}]$ and $[y_0: y_1: \cdots: y_{n_2}]$ are the homogeneous coordinates on $\CC P^{n_1}$, reps. $\CC P^{n_2}$.

We can define a natural circle action on $H_{n_1,n_2}^{\mathbf{0}}(d_1,1)$ 
$$\lambda\cdot [x_0: x_1: \cdots: x_{n_1}]=[x_0: \lambda x_1: \cdots: \lambda x_{n_1}] $$
$$\lambda\cdot [y_0: y_1: \cdots: y_{n_2}]=[y_0: \lambda^{-d_1} y_1: \cdots: \lambda^{-d_1} y_{n_2}]$$
where $\lambda\in S^1$.
\vskip.2cm
 We propose a natural question,
\vskip.2cm
{\bf Problem:}   Does there exist a non-trivial circle action on $H_{n_1,n_2}^{\bI}(d_1,1)$ for $\bI\neq {\bf 0}$?

\vskip.5cm

Denote $a_{i}(n)$ the coefficient in the dyadic expansion of $n\in\Z_+$:
$$n=a_{0}(n)+a_{1}(n)2^{1}+a_{2}(n)2^{2}+\cdots+a_{m}(n)2^{m}.$$

\begin{theorem}[\protect Lucas Theorem\cite{Lu}]
For any $k\in\Z$,
$$\binom{n+k}{n}\equiv\prod\limits_{i=0}^{m}\binom{a_{i}(n+k)}{a_{i}(n)}\ \mod\ 2.$$

\end{theorem}

And the following properties are based on Lucas Theorem.

\begin{proposition}[c.f. \cite{Zh96E}]
 $\binom{n+k}{n}\equiv 1\ \mod\ 2, n\in \N, k\in\Z$ iff
\begin{equation}
\begin{cases}a_{i}(n)+a_{i}(k)\leq 1,\forall\ i\geq 0, & \mathrm{ for}\ k\geq 0;\\
     a_{i}(-k-1-n)+ a_{i}(n)\leq 1, \forall\ i\geq 0, &\mathrm{ for}\ k< -n-1.
\end{cases}
\end{equation}

\end{proposition}

\begin{proof}
If $k\geq 0$, by Lucas Theorem,
$$\binom{n+k}{n} \equiv 1\ \ \mod\ 2 \iff \prod\limits_{i=0}^{m}\binom{a_{i}(n+k)} {a_{i}(n)}\equiv 1\ \ \mod\ 2,$$
 which equals to $$\forall \ 0\leq i\leq m,\binom{a_{i}(n+k)}{a_{i}(n)}=1,i.e.\ \forall\ 0\leq i\leq m,a_{i}(n+k)\geq a_{i}(n).$$

Since $k=(n+k)-n$,
 $$(n+k)=a_{0}(n+k)2^{0}+a_{1}(n+k)2^{1}+\cdots +a_{m}(n+k)2^{m}+\cdots,$$
 we have
 $$\sum\limits_{i=0}^{m}a_{i}(k)2^i=\sum\limits_{i=0}^{m}(a_{i}(n+k)-a_{i}(n))2^i.$$
 \indent
 In general, $a_{i}(k)\neq a_{i}(n+k)-a_{i}(n)$, but if $a_{i}(n+k)\geq a_{i}(n)$, and since $a_{i}(n+k)\leq 1$, we have $ 0\leq a_{i}(n+k)- a_{i}(n)\leq1$. And since the factorization is unique, therefore
 $$ 0\leq a_{i}(n+k)- a_{i}(n)\leq1     \iff  a_{i}(k)=a_{i}(n+k)-a_{i}(n)$$
 Thus
   $$\binom{n+k}{n}\equiv1 \ \mod\ 2 \iff \forall 0\leq i\leq k,a_{i}(k)+a_{i}(n)=a_{i}(n+k)\leq 1.$$

\indent
If $k<-n$, $$\binom{n+k}{n}=\frac{(n+k)\cdots(k+1)}{n!}=(-1)^{n}\frac{(-k-1)\cdots(-n-k)}{n!}=(-1)^{n}\binom{-k-1}n.$$

Thus $$\binom{n+k}{n}\equiv\binom{-k-1}{n}\equiv1 \ \mod\ 2 \iff \forall\ 0\leq i\leq k,  a_{i}(-k-1-n)+ a_{i}(n)\leq 1.
$$
\end{proof}

We give a useful tool to simplify our computation.

\begin{proposition}[Appendix \ref{8}]\label{16} For any nonnegative integers $m, n$, we have
\begin{equation}\label{combinotoric}
\binom{m+n}  {n}\equiv
\begin{cases}
0\ \ \ \mod\ 2, & n\cdot m\ \text{or}\ [\frac{n}{2}]\cdot[\frac{m}{2}] \text{is odd};\\
\binom{[\frac{n}{4}]+[\frac{m}{4}]} {[\frac{n}{4}]}\ \ \ \mod\ 2,&  \text{otherwise}.
\end{cases}
\end{equation}
\end{proposition}

\bigskip

\begin{theorem}[\protect Stolz  \cite{Sto}] \label{14}
Let $M$ be a simply connected, closed manifold of dimension$\geq 5$. Then $M$ does not admit a Riemannian metric of PSC iff $M$ is a spin manifold and $\alpha (M)\neq 0$.
\end{theorem}

 Since by \cite[Theorem A]{KM}, for any oriented manifold $M^{2n}$, any codimension 2 homology class is represented by a submanifold $K\subset M$, and $(M,K)$ is $n$-connected. In the following discussion, let us assume $H_{n_1,n_2}^{\bI}(d_1,d_2)$ to be simply connected.

Thus $H_{n_1,n_2}^{\bI}(d_1,d_2)$ as a spin submanifold of $V$ with dimension no less than than 5 does not admit a Riemannian metric of PSC iff $\alpha (M)\neq 0$.

\begin{corollary}
$H_{n_1,n_2}^{\bI}(1,1)$ always admits a Riemannian metric of PSC (Example \ref{3.3}).
In particular, Milnor hypersurfaces always admits a Riemannian metric of PSC (Example \ref{3.1}).
\end{corollary}

\begin{corollary}
 When $k_1=-\frac{n_1+1-d_1}{2},k_2=-\frac{n_2+1-d_2}{2}$ are integers, The spin hypersurface $H_{n_1,n_2}^{\bf 0}(d_1,d_2)$ of \ $\CC P^{n_1}\times\CC P^{n_2}$ does not admit a Riemannian metric of PSC if and only if $\ \forall\ i,\ a_i(k_l)+a_i(n_l)\leq 1, l=1,2$ (Example \ref{3.2})
\end{corollary}

Motivated by Zhang's work in \cite{Zh96E}, we give the sufficient and necessary conditions for existing positive scalar curvature on $H_{n_1,n_2}^{\bI}(d_1,d_2)$, which are described in Example \ref{3.5} and  \ref{3.6}.
\begin{corollary}\label{n_1=1}
Assume $ n_1+n_2\equiv 2\ \mod\ 4$,\ $n_1=1$ and $k_1=-\frac{n_1+1-d_1-\sigma_1}{2},k_2=-\frac{n_2+1-d_2}{2}$ are integers. A spin $H_{n_1,n_2}^{\bI}(d_1,d_2)$ doest not admit a Riemannian metric of PSC if and only if one of the following conditions satisfied
\begin{itemize}
  \item $k_{2}\geq 0$, $k_{2}\equiv 0\ \mod\ 4$, $k_{1}\equiv 0\ \mod\ 2$, $\forall\ i, a_{i}([\frac{k_{2}}{4}])+a_{i}([\frac{n_{2}}{4}])\leq 1$;
  \item $k_{2}\geq 0$, $k_{2}\equiv 1\ \mod\ 4$, $\sigma_{1}\equiv 1\ \mod\ 2$ , $\forall\ i, a_{i}([\frac{k_{2}}{4}])+a_{i}([\frac{n_{2}}{4}])\leq 1$;
  \item $k_{2}\geq 0$, $k_{2}\equiv 2\ \mod\ 4$, $k_{1}+\sigma_{1}\equiv 0\ \mod\ 2$, $\forall\ i, a_{i}([\frac{k_{2}}{4}])+a_{i}([\frac{n_{2}}{4}])\leq 1$;
  \item $k_{2}\leq -n_{2}-1$, $-k_{2}\equiv 0\ \mod\ 4$, $k_{1}\equiv 0\ \mod\ 2$, $\forall\ i,a_{i}([\frac{-k_{2}-1-n_{2}}{4}])+a_{i}([\frac{n_{2} }{4}])\leq 1$;
  \item $k_{2}\leq -n_{2}-1$, $-k_{2}\equiv 2\ \mod\ 4$, $k_{1}+\sigma_{1}\equiv 0\ \mod\ 2$, $\forall\ i, a_{i}([\frac{-k_{2}-1-n_{2}}{4}])+a_{i}([\frac{n_{2} }{4}])\leq 1$;
  \item $k_{2}\leq -n_{2}-1$, $-k_{2}\equiv 3\ \mod\ 4$, $\sigma_{1}\equiv 1\ \mod\ 2,\forall\ i, a_{i}([\frac{-k_{2}-1-n_{2}}{4}])+a_{i}([\frac{n_{2} }{4}])\leq 1$.
\end{itemize}
\end{corollary}
\begin{proof}
In the following we assume $n_{2}=4m+1$ for some $m\in\Z$. And since $H_{n_1,n_2}^{\bI}(d_1,d_2)$ is spin, $k_1,k_2\in\Z.$

And
$$\alpha(H_{n_1,n_2}^{\bI}(d_1,d_2))\equiv \binom{k_1+1}{1}\binom{n_2+k_2}{n_2}+\sigma_1\binom{n_2+k_2}{n_2+1}\ \mod\ 2.$$

\begin{description}
  \item[1] If $-n_{2}\leq k_{2}<0$, $\alpha =0$. Thus $H_{n_1,n_2}^{\bI}(d_1,d_2)$ admits a Riemannian metric of PSC.
\vskip.2cm
  \item[2] If $k_{2}\geq 0$, $\alpha(H_{n_1,n_2}^{\bI}(d_1,d_2))\equiv (k_{1}+1)\binom{n_2+k_2}{n_2}+\sigma_1\binom{n_2+k_2}{n_2+1}\ \mod\ 2$.

\begin{enumerate}
  \item If $k_{2}\equiv 0\ \mod\ 4$, assume $k_{2}=4k$, for some $k\in\Z$. Then\\

  $\binom{n_{2}+k_{2}} {n_{2}}=\binom{4m+1+4k} {4m+1}\equiv\binom{m+k} {k}\ \mod\ 2,$\\

  $ \binom{n_{2}+k_{2}} {n_{2}+1}=\binom{4m+1+4k} {4m+2}= \binom{(4m+2)+(4(k-1)+3)} {4m+2}\equiv 0\ \mod\ 2.$\\

Thus
$$\alpha(H_{n_1,n_2}^{\bI}(d_1,d_2))\equiv(k_{1}+1)\binom{m+k} {k}\ \mod\ 2.
$$
So $H_{n_1,n_2}^{\bI}(d_1,d_2)$  does not admit a Riemannian metric of PSC iff

$$
k_{1}\equiv 0\ \mod\ 2, \forall\ i, a_{i}([\frac{k_{2}}{4}])+a_{i}([\frac{n_{2}}{4}])\leq 1.
$$

  \item If $k_{2}\equiv 1\ \mod\ 4$, assume $k_{2}=4k+1$, for some $k\in\Z$.
  Thus
  $$\alpha(H_{n_1,n_2}^{\bI}(d_1,d_2))\equiv \sigma_{1}\binom{m+k} {k}\ \mod\ 2.$$
  So $H_{n_1,n_2}^{\bI}(d_1,d_2)$  does not admit a Riemannian metric of PSC iff
  $$
 \sigma_{1}\equiv 1\ \mod\ 2, \forall i, a_{i}([\frac{k_{2}}{4}])+a_{i}([\frac{n_{2}}{4}])\leq 1.
$$

 \item If $k_{2}\equiv 2 \ \mod\ 4$, assume $k_{2}=4k+2$, for some $k\in\Z$.
  Thus
  $$\alpha(H_{n_1,n_2}^{\bI}(d_1,d_2))\equiv (k_{1}+1+\sigma_{1})\binom{m+k} {k}\ \mod\ 2.$$
  So $H_{n_1,n_2}^{\bI}(d_1,d_2)$  does not admit a Riemannian metric of PSC iff
  $$
 (k_{1}+\sigma_{1})\equiv 0\ \mod\ 2, \forall i, a_{i}([\frac{k_{2}}{4}])+a_{i}([\frac{n_{2}}{4}])\leq 1.
$$

\vskip.2cm
\item If $k_{2}\equiv 3 \ \mod\ 4$, assume $k_{2}=4k+3$, for some $k\in\Z$.
  Thus
  $$\alpha(H_{n_1,n_2}^{\bI}(d_1,d_2))\equiv 0\ \mod\ 2.$$
  So $H_{n_1,n_2}^{\bI}(d_1,d_2)$ always admits a Riemannian metric of PSC.
\end{enumerate}

\vskip.3cm
  \item[3] If $k_{2}\leq -n_{2}-1$, the analysis is similar to $k_{2}\geq 0$.
\end{description}

\end{proof}



\begin{corollary} \label{n_1=2}
Assume $n_1+n_2\equiv 2\ \mod\ 4$,\ $n_1=2$ and $k_1=-\frac{n_1+1-d_1-\sigma_1}{2},k_2=-\frac{n_2+1-d_2}{2}$ are integers. A spin $H_{n_1,n_2}^{\bI}(d_1,d_2)$ doest not admit a Riemannian metric of PSC if and only if one of the following condition satisfied
\begin{itemize}
  \item $k_{2}\geq 0$, $k_{2}\equiv 0\ \mod\ 4$, $k_{1}\equiv 0\ or\ 1\ \mod\ 4$, $ \forall\ i, a_{i}([\frac{k_{2} }{4}])+a_{i}(\frac{n_{2}}{4})\leq 1$;
  \item $k_{2}\geq 0$, $k_{2}\equiv 1\ \mod\ 4$, $ \binom{k_1+2}{2}  +\frac{\sigma_1^2-2\sigma_2}{2}+\frac{(2k_1+3)\sigma_1}{2}  \equiv 1\  \mod\ 2,\forall\ i, a_{i}([\frac{k_{2} }{4}])+a_{i}(\frac{n_{2}}{4})\leq 1;$
  \item  $k_{2}\geq 0$, $k_{2}\equiv 2\ \mod\ 4$, $  \binom{k_1+2}{2} +\sigma_1^2-\sigma_2 \equiv 1\  \mod\ 2,\forall\ i, a_{i}([\frac{k_{2} }{4}])+a_{i}(\frac{n_{2}}{4})\leq 1;$
  \item $ k_{2}\geq 0$, $k_{2}\equiv 3\ \mod\ 4$, $ \binom{ k_{1}+2} {2}+\frac{\sigma_{1}(2k_{1}+3-\sigma_{1})}{2} \equiv 1\  \mod\ 2,\forall\ i, a_{i}([\frac{k_{2} }{4}])+a_{i}(\frac{n_{2}}{4})\leq 1;$
   \item $k_{2}= -n_{2}-1$, $ \frac{(k_{1}+1)(k_{1}+2)}{2}+\frac{\sigma_{1}(\sigma_{1}-2k_{1}-3)}{2} \equiv 1\  \mod\ 2;$

  \item $k_{2}\leq -n_{2}-2$, $-k_{2}\equiv 0\ \mod\ 4$, $ k_{1}\equiv 0\ or\ 1\ \mod\ 4,\forall\ i, a_{i}([\frac{-k_{2}-1-n_{2}}{4}])+a_{i}(\frac{n_{2}}{4})\leq 1;$
  \item $k_{2}\leq -n_{2}-2$, $-k_{2}\equiv 1\ \mod\ 4$, $ \binom{k_1+2}{2}+\frac{\sigma_{1}(\sigma_{1}-2k_{1}-3)}{2} \equiv  1\ \mod\ 2,\forall\ i, a_{i}([\frac{-k_{2}-1-n_{2}}{4}])+a_{i}(\frac{n_{2}}{4})\leq 1;
 $
  \item  $k_{2}\leq -n_{2}-2$, $-k_{2}\equiv 2\ \mod\ 4$, $ \binom{k_1+2}{2}+\sigma_1^2-\sigma_2 \equiv  1\ \mod\ 2,\forall \ i,a_{i}([\frac{-k_{2}-1-n_{2}}{4}])+a_{i}(\frac{n_{2}}{4})\leq 1;$
  \item$k_{2}\leq -n_{2}-2$, $-k_{2}\equiv 3\ \mod\ 4$, $ \binom{k_1+2}{2}+\frac{(2k_1+3+3\sigma_1)\sigma_1}{2}-\sigma_2 \equiv  1\ \mod\ 2,\forall\ i, a_{i}([\frac{-k_{2}-1-n_{2}}{4}])+a_{i}(\frac{n_{2}}{4})\leq 1;$
\end{itemize}
\end{corollary}

\begin{proof}

Assume $n_{2}=4m$, for some $m\in\Z$. The following discussion is similar to Corollary  \ref{n_1=1}.

\begin{description}
  \item[1] For $k_{2}\geq 0$:\\
\begin{align*}
\alpha(H_{n_1,n_2}^{\bI}(d_1,d_2))\equiv& \binom{k_1+2}{2}\binom{4m+k_2}{4m}+(-\frac{\sigma_1^2-2\sigma_2}{2}+\frac{(2k_1+3)\sigma_1}{2})\binom{4m+k_2}{4m+1}\\
&+(\sigma_1^2-2\sigma_2)\binom{4m+k_2+1}{4m+2}+\sigma_2\binom{4m+k_2}{4m+2}\ \mod\ 2
\end{align*}

   \begin{enumerate}
  \item If $k_{2}\equiv 0\ \mod\ 4$, assume $k_{2}=4k$, then
 \begin{small}
\begin{align*}
 \alpha(H_{n_1,n_2}^{\bI}(d_1,d_2))\equiv & \binom{k_1+2}{2}\binom{4m+4k}{4m}+(-\frac{\sigma_1^2-2\sigma_2}{2}+\frac{(2k_1+3)\sigma_1}{2})\binom{4m+4k}{4m+1}\\
&+(\sigma_1^2-2\sigma_2)\binom{4m+4k+1}{4m+2}+\sigma_2\binom{4m+4k}{4m+2}\ \mod\ 2\\
\equiv& \binom{k_1+2}{2}\binom{m+ k}{ m}+(-\frac{\sigma_1^2-2\sigma_2}{2}+\frac{(2k_1+3)\sigma_1}{2})\binom{4m+1+4k-1}{4m+1}\\
&+(\sigma_1^2-2\sigma_2)\binom{4m+2+4(k-1)+3}{4m+2}+\sigma_2\binom{4m+2+4k-2}{4m+2}\ \mod\ 2\\
\equiv& \binom{k_1+2}{2}\binom{m+ k}{ m} \mod\ 2.
\end{align*}\end{small}
And $\binom{2+k_{1}}{2}\equiv 1\ \mod\ 2$   equal to $ a_1(k_{1})=0$, more precisely, $k_{1}\equiv 0\ or\ 1\ \mod\ 4$.
\vskip.2cm
So $H_{n_1,n_2}^{\bI}(d_1,d_2)$ does not admit a Riemannian metric of PSC iff
  $$ k_{1}\equiv 0\ or\ 1\ \mod\ 4,\forall\ i, a_{i}([\frac{k_{2} }{4}])+a_{i}(\frac{n_{2}}{4})\leq 1.$$

 \item If $k_{2}\equiv 1\ \mod\ 4$, assume $k_{2}=4k+1$, then\\

 $\alpha(H_{n_1,n_2}^{\bI}(d_1,d_2))\equiv \{\binom{k_1+2}{2}  +\frac{\sigma_1^2-2\sigma_2}{2}+\frac{(2k_1+3)\sigma_1}{2}  \}\binom{m+k}{k}\ \mod\ 2.
$\\

 So $H_{n_1,n_2}^{\bI}(d_1,d_2)$  does not admit a Riemannian metric of PSC iff
 $$\binom{k_1+2}{2}  +\frac{\sigma_1^2-2\sigma_2}{2}+\frac{(2k_1+3)\sigma_1}{2} \equiv 1\  \mod\ 2,\forall i, a_{i}([\frac{k_{2} }{4}])+a_{i}(\frac{n_{2}}{4})\leq 1.$$

  \item If $k_{2}\equiv 2\ \mod\ 4$, assume $k_{2}=4k+2$, then\\

 $\alpha(H_{n_1,n_2}^{\bI}(d_1,d_2))\equiv(\binom{k_1+2}{2} +\sigma_1^2-\sigma_2)\binom{ m+ k }{ m }\ \mod\ 2.$\\

  \item If $k_{2}\equiv 3\ \mod\ 4$, assume $k_{2}=4k+3$, then\\

 $\alpha(H_{n_1,n_2}^{\bI}(d_1,d_2))\equiv (\binom{k_{1}+2} {2}+\frac{\sigma_{1}(2k_{1}+3-\sigma_{1})}{2})\binom{m+k}{m}\ \mod\ 2.$\\

\end{enumerate}

  \item[2]
If $-n_{2}\leq k_{2}< 0$,  $H_{n_1,n_2}^{\bI}(d_1,d_2)$ always admits a Riemannian metric of PSC.

\vskip.2cm
  \item[3]
If $k_{2}=-n_{2}-1$,
then

$\alpha(H_{n_1,n_2}^{\bI}(d_1,d_2))\equiv  \frac{(k_{1}+1)(k_{1}+2)}{2}+\frac{\sigma_{1}(\sigma_{1}-2k_{1}-3)}{2}\ \mod\ 2.
$\\

So $H_{n_1,n_2}^{\bI}(d_1,d_2)$  does not admit a Riemannian metric of PSC iff
$$\frac{(k_{1}+1)(k_{1}+2)}{2}+\frac{\sigma_{1}(\sigma_{1}-2k_{1}-3)}{2}\equiv 1\  \mod\ 2.$$

  \item[4]
If $k_{2}\leq -n_{2}-2$, the analysis  is similar to $k_{2}\geq 0$.
\end{description}

\end{proof}


\appendix

\section{Relations between $A(n,l)$ and some classical numbers}
The numbers
\begin{equation*}
A(n,l)=
\begin{cases}
\frac{1}{n!}\sum\limits_{m=0}^{l}(-1)^{l-m}\binom{l }{m}m^{n }, &  0\leq l\leq n;\\
0,&  \text{otherwise.}
\end{cases}
\end{equation*}
appear in Section \ref{A} has a generating function
  $$e^{y(e^{x}-1)}=\sum_{l,n\geq 0}A(n,l)x^{n}\frac{y^{l}}{l!}$$
and $A(n,l)$ can be obtained recursively by $ A(n,l)=\frac{l}{n}(A(n-1,l)+A(n-1,l-1))$ and $A(0,0)=1$.
\vskip.2cm
\begin{lemma}\label{4}\label{3}
$A(n,l)$ has the following relations with the Stirling number, Bell number, Bernoulli number (c.f. \cite{CG}), and divided difference (c.f. \cite{BF}).
\begin{enumerate}
\item [$(1)$] Let $S(n,l)$ be the Stirling number of the second kind which is generated by
   $$e^{y(e^x-1)}=\sum_{l,n\geq 0}S(n,l)y^l\frac{x^n}{n!}, $$
   we have
$$S(n,l)=\frac{n!}{l!}A( n,l)$$
and  $$\sum_{n=l}^{\infty}A(n,l)u^n=(e^{u}-1)^l.$$

\item [$(2)$] Let $B_n$ be  the  Bell number, we have
$$B_n=\sum_{l=0}^{n}\frac{n!}{l!}A(n,l).$$

\item [$(3)$] Let $B_{n}(0)$ be the Bernoulli number  which is generated by
 $$\frac{x}{e^{x}-1}=\sum_{n=0}^{\infty}B_n(0)\frac{x^n}{n!},$$ we have
$$B_{n}(0)=\sum_{l=0}^{n}\frac{n!}{l+1}A( n,l).$$

\item [$(4)$] Let $P_{n}[0,1,2,\cdots,l]$ be the $l$-th divided difference for $P_{n}(x)=x^n$, we have
$$P_{n }[0,1,2,\cdots,l] =\frac{n!}{l!}A(n,l).$$
\end{enumerate}
\end{lemma}
\begin{proof}

$(1)$:
 Stirling number of second kind is of recurrence relation:$$S(n,l)=lS(n-1,l)+S(n-1,l-1).$$
  Thus $$\frac{l!}{n!}S(n,l)=\frac{l}{n}\{\frac{l!}{(n-1)!}S(n-1,l)+\frac{(l-1)!}{(n-1)!}S(n-1,l-1)\}.$$
  Since $A(n,l)$ is of the same recurrence with $\frac{l!}{n!}S(n,l)$, thus $$A(n,l)=\frac{l!}{n!}S(n,l).$$

  And by definition, $$A(n,l)=\frac{l!}{n!}S(n,l)=\frac{l!}{n!}\frac{1}{l!}\sum_{m=0}^{l}(-1)^m\binom{l}{m}(l-m)^{n}=\frac{1}{n!}\sum_{m=0}^{l}(-1)^{l-m}\binom{l }{m}m^{n }.$$

So
\begin{align*}
&\sum\limits_{n=l}^{\infty}A(n,l)u^n\\
=&\sum\limits_{n=l}^{\infty}\frac{1}{n!}\sum\limits_{m=0}^{l}(-1)^{l-m}\binom{l }{m}m^{n }u^n\\
=&\sum\limits_{m=0}^{l}(-1)^{l-m}\binom{l}{m}\sum\limits_{n=l}^{\infty} \frac{(mu)^{n}}{n!}\\
=&\sum\limits_{m=0}^{l}(-1)^{l-m}\binom{l}{m}(e^{mu}-\sum\limits_{n=0}^{l-1} \frac{(mu)^{n}}{n!})\\
\end{align*}
\begin{align*}
=&\sum\limits_{m=0}^{l}(-1)^{l-m}\binom{l}{m}(e^{u})^m-\sum_{m=0}^{l}(-1)^{l-m}\binom{l}{m}\sum\limits_{n=0}^{l-1} \frac{(mu)^{n}}{n!}\\
=&(e^u-1)^l-\sum\limits_{n=0}^{l-1}\frac{u^n}{n!}\sum\limits_{m=0}^{l}(-1)^{l-m}\binom{l}{m}m^n\\
=&(e^u-1)^l-\sum\limits_{n=0}^{l-1}\frac{u^n}{n!}S(n,l)\\
=&(e^{ u}-1)^l.
\end{align*}

Note that $S(n,l)=0$ for $n<l$.\\

$(2)$:  a simple corollary of (1),$$B_n=\sum_{l=0}^{n}S(n,l)=\sum_{l=0}^{n}\frac{n!}{l!}A(n,l).$$

$(3)$:
$$\sum_{l,n\geq 0}\frac{(-1)^l}{l+1}A(n,l)x^n=\sum_{l=0}^{\infty}\frac{(-1)^l}{l+1}\sum_{n=l}^{\infty}A(n,l)x^n=\sum_{l=0}^{\infty}\frac{(-1)^l}{l+1}(e^x-1)^l=\frac{x}{e^x-1}.$$

$(4)$:
The divided difference of $P_n(x)=x^n$ is
$$P_{n}[x_{0},x_{1},\cdots,x_{l}]=\frac{1}{l!}\sum_{m=0}^{l}(-1)^{n} \binom{l}{m}x_{l-m}^n .$$
And
\begin{equation*}
P_{n}[x_{0},x_{1},\cdots,x_{l}]=
\begin{cases}
0 & l>n\\
1&  l=n\\
\sum\limits_{0\leq t_1\leq t_2\leq \cdots\leq t_{n-l}\leq l} x_{t_1}\cdots x_{t_{n-l}}& l<n
\end{cases}
\end{equation*}

By induction:
\begin{itemize}
  \item Claim 1: $A(l,l)=\frac{l!}{l!}P_{l }[0,1,2,\cdots,l]=1$ is true.
\vskip.2cm
  By the recurrence, $A(1,1)=1=\frac{1!}{1!}P_1[0,1]$ is true and $A(n,l)=0$ for $l>n\ \text{or}\ l<0$. Assume $A(l,l)=1$ is true for any integer less than $n$.\\
  \indent

  $A(n,n)=\frac{n}{n}(A( n-1,n)+A(n-1,n-1))=A(n-1,n-1)=\cdots=A(1,1).$\\
  \indent

  So the claim is proved.

  \item Claim 2: $A(n,l)=\frac{l!}{n!}P_{n }[0,1,2,\cdots,l]$ is true.
\vskip.2cm
By claim 1, $A(l,l)$ is true,  assume $A(n,l)$ is true for all integer $l$ and $n$ less than $l+1 $ resp $n+1$.

\begin{align*}
&A( n+1,l)\\
=&\frac{l}{n+1}(A( n,l)+A(n,l-1))\\
=&\frac{l}{n+1}(\frac{l!}{n!}P_{n }[0,1,2,\cdots,l]+\frac{(l-1)!}{n!}P_{n }[0,1,2,\cdots,l-1])\\
=&\frac{l!}{(n+1)!} \sum\limits_{\substack{1\leq i\leq n-l\\t_i\in\{0,1,2\cdots,l\}\\t_1\leq t_2\leq \cdots\leq t_{n-l}}}
 t_1\cdots t_{n-l}l+\frac{l!}{(n+1)!}\sum_{\substack{1\leq i\leq n+1-l\\t_i\in\{0,1,2\cdots,l-1\}\\t_1\leq t_2\leq \cdots\leq t_{n+1-l}}}
 t_1\cdots t_{n+1-l}\\
  =&\frac{l!}{(n+1)!}
\{ \sum\limits_{\substack{1\leq i\leq n-l\\t_i\in\{0,1,2\cdots,l\}\\t_1\leq t_2\leq \cdots\leq  t_{n+1-l}=l}}
 t_1\cdots t_{n-l}t_{n+1-l}+\sum_{\substack{1\leq i\leq n+1-l\\t_i\in\{0,1,2\cdots,l \}\\t_1\leq t_2\leq \cdots\leq  t_{n+1-l}< l}} t_1\cdots t_{n+1-l}\}\\
  =&\displaystyle\frac{l!}{(n+1)!}\sum\limits_{\substack{t_i\in\{1,2\cdots,l\}\\t_1\leq t_2\leq \cdots\leq  t_{n+1-l}\leq l}} t_1\cdots t_{n-l}t_{n+1-l}\\
 =&\frac{l!}{(n+1)!}P_{n+1}[0,1,2,\cdots,l].
\end{align*}
\end{itemize}

\end{proof}

\section{Some combinatorial properties}

Denote $a_i(n)$ the coefficient in dyadic expansion of $n$:
$$
n=a_{0}(n)+a_{1}(n)2^{1}+a_{2}(n)2^{2}+\cdots+a_{i}(n)2^{i},
$$
where $i\in \Z$ must be finite integer.


The following properties is based on the Lucas Theorem.

\begin{proposition}\label{18}
For $m, n\in \N$, $$\binom{4m+n} {4m}\equiv \binom{m+[\frac{n}{4}]} {m}\ \mod\ 2. $$
\end{proposition}
\begin{proof}
If $n\equiv 0\ \mod\ 4$. Assume $n=4k$,
\begin{equation*}
\begin{aligned} 
\binom{4m+4k} {4m}
&=\frac{(4m+4k-1)!!}{(4m-1)!!(4k-1)!!}\frac{(4m+4k)!!}{(4k)!!(4m)!!}\\
&=\frac{(4m+4k-1)!!}{(4m-1)!!(4k-1)!!}\frac{(2m+2k)!}{(2k)!(2m)!}\\
&=\frac{(4m+4k-1)!!}{(4m-1)!!(4k-1)!!}\frac{(2m+2k-1)!!}{(2m-1)!!(2k-1)!!}\frac{(2m+2k)!!}{(2k)!!(2m)!!}\\
&=\frac{(4m+4k-1)!!}{(4m-1)!!(4k-1)!!}\frac{(2m+2k-1)!!}{(2m-1)!!(2k-1)!!}\frac{(m+k)!}{k!m!}\\
&\equiv \binom{m+k} {m}\ \mod\ 2.
\end{aligned}
\end{equation*}
For the rest dimensions, the analysis is similar.

\end{proof}

\begin{proposition}\label{11}
For $m, n\in \N$,
\begin{equation*}
\binom{4m+1+n}  {4m+1}\equiv
\begin{cases}
0\ \ \ \mod\ 2, & n \ \text{is odd};\\
\binom{m+[\frac{n}{4}]} {m}\ \ \ \mod\ 2,&  n\  \text{is even}.
\end{cases}
\end{equation*}
\end{proposition}
\begin{proof} When $n,4m+1$ are both odd,  $a_{0}(m)=a_{0}(n)=1$, thus $\binom{4m+1+n}{n}$ is even.

When $n=4k$,  $$\binom{4m+1+n} {4m+1}=\binom{4k+4m+1} {4k}\equiv \binom{m+k} {k}\equiv \binom{m+k} {m}\ \mod\ 2.$$

When $n=4k+2$,
\begin{equation*}
\begin{aligned}
\binom{4m+1+4k+2} {4m+1}
=& \frac{(4m+1+4k+2)!!}{(4m+1)!!(4k+1)!!}\frac{(4m+4k+2)!!}{(4m)!!(4k+2)!!}\\
=& \frac{(4m+1+4k+2)!!}{(4m+1)!!(4k+1)!!}\frac{(2m+2k+1)! }{(2m)!(2k+1)! }\\
=& \frac{(4m+1+4k+2)!!}{(4m+1)!!(4k+1)!!}\frac{(2m+2k+1)!! }{(2m-1)!!(2k+1)!! }\frac{(m+k )! }{(m)!(k )! }\\
\equiv & \binom{m+k} {m}\ \mod\ 2.
\end{aligned}
\end{equation*}

\end{proof}

\begin{corollary}\label{8}
For $m, n\in \N$,
\begin{equation*}
\binom{m+n}  {n}\equiv
\begin{cases}
0\ \ \ \mod\ 2, & n\cdot m\ \text{or}\ [\frac{n}{2}]\cdot[\frac{m}{2}] \text{is odd};\\
\binom{[\frac{n}{4}]+[\frac{m}{4}]} {[\frac{n}{4}]}\ \ \ \mod\ 2,&  \text{otherwise}.
\end{cases}
\end{equation*}
\end{corollary}
\begin{proof} When $n\cdot m$ is odd,$n,m$ are both odd, and since $a_{0}(m)=a_{0}(n)=1$, $\binom{m+n}{n}$ is even.
\\
\indent
When $[\frac{n}{2}][\frac{m}{2}]$ is odd, $n,m$ are both one of element in$\{4l_{1}+2,4l_{2}+3\}$, where $l_{1},l_{2}\in \N$. Since  $a_{1}(n)=a_{1}(m)=1$. Thus $\binom{m+n}{n}$ is even.
Otherwise
\begin{itemize}
  \item If $ n=4k$,$ \binom{m+n} {n}=\binom{4k+m} {m}\equiv \binom{k+[\frac{m}{4}]} {[\frac{m}{4}]}\equiv \binom{[\frac{m}{4}]+[\frac{n}{4}]} {[\frac{n}{4}]}\ \mod\ 2$.
  \item If $ n=4k+1$ and we only need to consider $m$  is even, since $n,m$ are both odd  discussed about above.
  \\$ \binom{m+n} {n}=\binom{4k+1+m} {m}\equiv \binom{[\frac{m}{4}]+[\frac{n}{4}]} {[\frac{n}{4}]}\ \mod\ 2$ by  proposition \ref{11}.
  \item If $ n=4k+2$ and we only need to consider $m=4l$ or $4l+1$, since the cases $n,m$ are elements in$\{4k_{1}+2,4k_{2}+3\}$ have been discussed about above.
   By  proposition \ref{11} and \ref{18}.

  $\displaystyle \binom{m+n} {n}=\binom{m+n} {m}=\binom{4k+2+4l} {4l} \equiv \binom{k+l} {k}\ \mod\ 2.$

  $\displaystyle \binom{m+n} {n}=\binom{m+n} {m}=\binom{4k+2+4l+1} {4l+1}\equiv \binom{k+l} {k}\ \mod\ 2.$
  \item If $ n=4k+3$ and we only need to consider $m=4l$. By proposition \ref{11},

  $\displaystyle \binom{m+n} {n}=\binom{m+n} {m}=\binom{4k+3+4l} {4l}\equiv \binom{k+l} {k}\ \mod\ 2.$
\end{itemize}
\end{proof}


\bibliographystyle{amsalpha}

\begin{thebibliography}{100}

\bibitem{AH70} M. Atiyah, F. Hirzebruch, {\em Spin-manifolds and group actions},  Essays on Topology and Related Topics (Memoires dedies Georges de Rham) Springer, Berlin, Heidelberg, 1970:18--28.

\bibitem{Bar} D. Baraglia, {\em The alpha invariant of complete intersections}, arXiv:2002.06750.

\bibitem{BGV} N. Berline,  E. Getzler, M. Vergne, {\em Heat kernels and Dirac operators}. Springer Science and Business Media, 2003.

\bibitem{BH58} A. Borel, F. Hirzebruch, {\em characteristic classes and homogeneous spaces}, I, Amer. J. Math. 1958, 80(2): 458--538.

\bibitem{Bo66}  J. M. Boardman, {\em Stable homotopy theory}, Mimeographed notes, Warwick, 1966.

\bibitem{Br83} R. Brooks, {\em The $\widehat{A}$-genus of complex hypersurfaces and complete intersections}. Proceedings of the American Mathematical Society, 1983, 87(3): 528--532.

\bibitem{BP}V. M. Buchstaber, T. E. Panov, {\em Toric topology}, Mathematical Suveys and Monographs 204, Amer. Math. Soc, Providence, RI(2015) MR2337880.

\bibitem{BF} R. L. Burden, J. D. Faires, {\em Numerical analysis}, Brooks/Cole,  9th edition, 2011.

\bibitem{CMS}  S. Choi, M. Masuda ,  D. Y. Suh, {\em Topological classification of generalized Bott towers}, Transactions of the American Mathematical Society, 2009, 362(02):1097--1112.
\bibitem{CG}   J. H. Conway, R. Guy, {\em The book of numbers}. Springer Science and Business Media, 1998.
\bibitem{DJ} M. W. Davis,  T. Januszkiewicz, {\em Convex polytopes, Coxeter orbifolds and torus actions}, Duke Mathematical Journal, 1991, 62(2):417--451.

\bibitem{FS} F. Fang, P. Shao, {\em Complete intersections with metrics of positive scalar curvature}, Comptes Rendus Mathematique, 2009, 347(13-14):797--800.


\bibitem{FZ} H. Feng, B. Zhang, {\em Existence of Riemannian metrics with positive scalar curvature of complex complete intersections}, Adv. Math. (China), 2007, 36:  47--50.

\bibitem{GL}M. Gromov, H. B. Lawson, {\em The classification of simply connected manifolds of positive scalar curvature}, Ann. Math. 1980: 423--434.

\bibitem{GK} M. Grossberg, Y, Karshon, {\em Bott towers, complete integrability, and the extended
character of representations}. Duke Math. 1994, 76: 23--58.

\bibitem{HBJ94} F. Hirzebruch, T. Berger and R. Jung, {\em Manifolds and modular forms}, Braunschweig: Vieweg, 1992.

\bibitem{Hir66} F. Hirzebruch, {\em Topological Methods in Algebraic Geometry} 3rd Ed, Springer-Verlag, 1966.

\bibitem{Hit} N. Hitchin, {\em Harmonic spinors}, Adv. Math. 1974, 14: 1--55.

\bibitem{KM} M. Kato, Y. Matsumoto, {\em Simply connected surgery of submanifolds in codimension two}, I. J. Math. Soc. Japan 1972, 24: 586--608.

\bibitem{KT} R. C. Kirby and L. R. Taylor, {\em Pin structures on low-dimensional manifolds},
in Geometry of Low-dimensional Manifolds,  Ed. S.K. Donaldson and C.B. Thomas, Cambridge Univ. Press. 1990,  2: 177--242.

\bibitem{Le} S. Lefschetz, {\em L'analysis situs et la g\'eom\'etrie alg\'ebrique}. Paris, 1924.

\bibitem{Li} A. Lichnerowicz, {\em Spineurs harmoniques}, C. R. Acad. Sci. Paris 1963, 257: 7--9.

\bibitem{LM} H. B. Lawson, M. L. Michelson, {\em Spin Geometry}, Princeton University Press, Princeton, New Jersey, 1994.

\bibitem{Lu} E. Lucas,  {\em Th\'eorie des Fonctions Num\'eriques Simplement P\'eriodiques}. American Journal of Mathematics. 1878, 1: 184--196.

\bibitem{LP}Z. L\"u, T. Panov,  {\em On toric generators in the unitary and special unitary bordism rings}, Algebraic and Geometry Topology, 2016, 16(5): 2865--2893.

\bibitem{SpE}  E. H. Spanier, {\em Algebraic topology}, Springer Science, Business Media, 1989.

\bibitem{Sto} S. Stolz, {\em Simply connected manifolds of positive scalar curvature}, Ann. Math, 1992, 136(2): 511--540.

\bibitem{So} G. Solomadin, {\em Explicit constructions of bordism of Milnor hypersurfaces $H_{1, n}$ and $\CC P^1\times \CC P^{n-1}$},  J. Math. Soc. Japan, Vol. 72, No 3 (2020), pp. 765-776.

\bibitem{Zh93} W. Zhang, {\em Spin$^c$-manifolds and Rokhlin congruences}, Comptes rendus de l'Acad\'{e}mie des sciences. S\'{e}rie 1, Math\'{e}matique, 1993, 317: 689--692.
\bibitem{Zh94} W. Zhang, {\em Circle bundles, adiabatic limits of $\eta $-invariants and Rokhlin congruences} Annales de l'institut Fourier. 1994, 44(1): 249--270.
\bibitem{Zh96R} W. Zhang, {\em Cobordism and Rokhlin congruences}, Acta Math. Scientia, 2009, 29B: 609--612.

\bibitem{Zh96E} W. Zhang, {\em Existence of Riemannian metrics with positive scalar curvature on complex hypersurfaces}, Acta Math. Sinica (Chin. Ser.) 1996, 39 (4): 460--462.





\end{thebibliography}

\end{document}